\documentclass[11pt]{article}

    \usepackage{amsmath}
    \usepackage{amsfonts,color,amssymb,mathrsfs,stmaryrd}
    \usepackage{amsthm}
    \usepackage{amssymb}
    \usepackage[normalem]{ulem}
    \usepackage{enumerate}
	\usepackage{hyperref}

\usepackage{graphicx}
\usepackage{wrapfig}
\usepackage{mathrsfs}
    \newcommand{\eps}{\varepsilon}

\newcommand{\PP}{\mathbb{P}}
\newcommand{\XX}{\mathbb{X}}
\def\Q{\mathbb{Q}}

\renewcommand{\H}{{\cal H}}

\newcommand{\0}{{\bf 0}}

\numberwithin{equation}{section}

    \newtheorem{thm}{Theorem}\numberwithin{thm}{section}

    \newtheorem{prop}[thm]{Proposition}
    \newtheorem{lem}[thm]{Lemma}
    \newtheorem{defn}{Definition}

        \def\N{\mathbb{N}}
    
    \def\E{\mathbb{E}}
    
    \def\0{{\bf 0}}
    
    \def\R{\mathbb{R}}
    
    \def\PP{\mathbb{P}}

    \renewcommand{\E}{\mathbb E \,}

    \newcommand{\MM}{\mathbb{M}}
    
    \newcommand{\tod}{\stackrel{{\cal D}}{\longrightarrow}}
    \newcommand{\todl}{\mathop{\stackrel{\cal D}{\longrightarrow}}\limits_{\la \to \infty}}
    \newcommand{\eqd}{\stackrel{{\cal D}}{=}}

    \renewcommand{\P}{{{\cal P}}}
    \newcommand\Pm{{\cal P}}

    \newcommand\etam{\eta}

    \DeclareMathOperator{\Var}{Var}
    \DeclareMathOperator{\Cov}{Cov}

    \newcommand{\Vol}{{\rm Vol}}

    \def\bdm{\begin{displaymath}}
    \newcommand{\edm}{\end{displaymath}}
    \def\benu{\begin{enumerate}}
    \def\eenu{\end{enumerate}}
    \def\beqn{\begin{equation}}
    \def\eeqn{\end{equation}}
    \def\be{\begin{equation}}
    \def\ee{\end{equation}}
    \def\bea{\begin{eqnarray}}
    \def\eea{\end{eqnarray}}
    \newcommand{\bean}{\begin{eqnarray*}}
    \newcommand{\eean}{\end{eqnarray*}}
    \newcommand{\bear}{\begin{eqnarray}}
    \newcommand{\eear}{\end{eqnarray}}
    
    \newcommand{\1}{{\bf 1}}

    \def\R{\mathbb{R}}

		\def\dint{\textup{d}}

    \def\hR{\hat{\R^d}}

   \def\hnu{\hat{\nu}}

     \def\hx{\hat{x}}
     \def\hy{\hat{y}}
     \def\hz{\hat{z}}
     \def\hW{\hat{W}}

    \def\la{{\lambda}}

    \def\qed{\hfill\hbox{${\vcenter{\vbox{
        \hrule height 0.4pt\hbox{\vrule width 0.4pt height 6pt
        \kern5pt\vrule width 0.4pt}\hrule height 0.4pt}}}$}}

\def\C{{C^{\rho}}}

\numberwithin{equation}{section}  % If you number theorems, etc. within sections,
\newcommand{\footremember}[2]{%
    \footnote{#2}
    \newcounter{#1}
    \setcounter{#1}{\value{footnote}}%
}
 
\title{\bf {Limit theory for unbiased and consistent \\ estimators of statistics of random tessellations}}
\author{%
  Daniela Flimmel\footremember{d}{Charles University, daniela.flimmel@karlin.mff.cuni.cz}%
  , Zbyn\v{e}k Pawlas\footremember{z}{Charles University, pawlas@karlin.mff.cuni.cz}%
  , Joseph E. Yukich\footremember{j}{Lehigh University, jey0@lehigh.edu}
  }
 \date{March 15, 2019}

\begin{document}

  \maketitle

\begin{abstract}
% text of abstract goes here!
We observe a realization of a stationary generalized weighted Voronoi tessellation of the d-dimensional Euclidean space within a bounded observation window. Given a geometric characteristic of the typical cell, we use the minus-sampling technique to construct an unbiased estimator of the average value of this geometric characteristic. Under mild conditions on the weights of the cells, we establish variance asymptotics and the asymptotic normality of the unbiased estimator as the observation window tends to the whole space. Moreover, the weak consistency is shown for this estimator.
\end{abstract}


\vspace{0.3cm}

\section{Introduction}\label{intro} % Initial capital letter, then lower case. No full stop.

% Write the text of your paper using normal LaTeX commands.
% For instance, you can use the `\cite' command~\cite{ref1}.
% When giving citations a numbering system is preferred~\cite{ref2},
% but an author--date system is also acceptable~\cite{ref3}.

%\subsection{Subsection} % Initial capital letter, then lower case. No full stop.

%\subsubsection{Subsubsection.} % Initial capital letter, then lower case. Full stop.

% If you write a theorem, lemma, proposition etc please use the
% appropriate environments. For instance:

Random tessellations are an  important model in stochastic geometry \cite{Chiu,SW08} and they have numerous applications in engineering and the natural sciences \cite{OBS}.  This paper focuses on random Voronoi tessellations of $\R^d$ as well as the so-called generalized weighted Voronoi tessellations. We shall be interested in developing the limit theory for unbiased and consistent estimators of statistics of a typical cell in a generalized weighted Voronoi tessellation.

The estimators are constructed by observing the tessellation within a bounded window.  Unbiased estimators are constructed by considering only those cells which lie within the bounded
window.  This technique, known as minus-sampling, has a long history going back to Miles \cite{M74} as well as Horvitz and Thompson; see \cite{B99} for details. In this paper we use stabilization methods to develop expectation and variance asymptotics, as well as central limit theorems, for unbiased and asymptotically consistent estimators of geometric statistics of a typical cell.

Generalized weighted Voronoi tessellations are defined as follows.  Let $\P$ be a unit intensity stationary point process on $\R^d$.
The points of $\P$ carry independent marks in the space $\MM \subseteq \R^+$ and follow the probability law
$\Q_\MM$.  Thus the atoms of $\P$ belong to $\R^d \times \MM$.  The elements of $\R^d \times \MM$ will be denoted by $\hx := (x,m_x)$.  To define weighted  Voronoi tessellations we introduce
a {\em weight function} $\rho: \R^d \times (\R^d \times \MM) \to \R$ which for each $\hx \in \Pm$ generates the weighted cell
$$
\C(\hx,\Pm) := \left\{ y \in \R^d: \rho(y,\hx) \leq \rho(y,\hz)\;\text{for all}\; \hz \in \Pm\right\}.
$$
Letting $\|x\|$ denote the Euclidean norm of $x$, we focus on the following well-known weights:
\begin{enumerate}
\item[ (i)] Voronoi cell: $\rho_1(y,\hx) := \|x-y\|$,
\item[(ii)] Laguerre cell: $\rho_2(y,\hx) := \|x-y\|^2 - m_x^2$,
\item[(iii)] Johnson--Mehl cell: $\rho_3(y,\hx) := \|x-y\| - m_x$.
\end{enumerate}
Notice that larger values of $m_x$ generate larger cells $\C(\hx,\Pm)$.  Voronoi and Laguerre cells are convex whereas the Johnson-Mehl cells need not be convex. The weight functions  $\rho_i( \cdot, \hx), i = 1,2,3$  generate the Voronoi, Laguerre \cite{LZ}, and Johnson--Mehl tessellations \cite{M92},
respectively and are often called the power of the point $x$. When $\Pm$ is a Poisson point process we shall refer to these tessellations as generalized Poisson--Voronoi weighted tessellations.

Denote by $K^{\rho}_{\0}:= K^{\rho}_{\0}(\Pm)$ the typical cell of a random tessellation defined by the weight $\rho$ and generated by $\Pm$. We denote by $Q^{\rho}$ the distribution of
the typical cell. For a formal definition of the typical cell see e.g.~\cite[Chapter 10]{SW08}.  Denote by ${\bf F}^{d}$
the space of all closed subsets of $\R^d$ and let $h: {\bf F}^d \rightarrow \R$ describe a  geometric characteristic of elements of ${\bf F}^{d}$ (e.g. diameter, volume). We have
two goals: (i) use minus-sampling to construct unbiased estimators of $\E h(K^{\rho}_\0) = \int h(K) \, Q^{\rho} (\dint K)$
and (ii) establish variance asymptotics and asymptotic normality of such
estimators.  As a by-product, we also establish the limit theory for geometric statistics of Laguerre and Johnson--Mehl tessellations, adding to the results of \cite{Pe07, PY1} which are confined to Voronoi tessellations.

\section{Main results}
Let $(\Omega, \mathcal{F}, \PP)$ be the common probability space and let $(\MM, \mathcal{F}_{\MM}, \mathbb{Q}_{\MM})$ be the mark space. Denote by $\hR$ the Cartesian product of $\R^d$
and $\MM$ and by $\hat{\mathcal{F}}$ the product $\sigma$-algebra of $\mathcal{B}(\R^d)$ and $\mathcal{F}_\MM$.
Let ${\bf N}$ be the set of all locally finite marked counting measures on $\hR$. An element of ${\bf N}$ can be interpreted as a marked point configuration. Therefore, we treat
it as a set in the notation. The set ${\bf N}$ is equipped with the standard $\sigma$-algebra $\mathcal{N}$ which is the smallest $\sigma$-algebra such that all  mappings $\pi_A:
{\bf N} \to \N \cup \{0,\infty\}, \P \mapsto \P(A), A \in \hat{\mathcal{F}}$, are measurable.

Define for all $z, x \in \R^d$
$$
C^{\rho}_z(\hx, \Pm) := \C(\hx, \Pm) + (z - x).
$$
Thus $\C(\hx, \Pm) = x + C^{\rho}_{\0}(\hx, \Pm)$ where $\0$ denotes a point at the origin of $\R^d$.

Recall that  $h: {\bf F}^d \to \R$ measures a geometric characteristic of elements of ${\bf F}^{d}$. We assume that $h$ is invariant with respect to shifts, namely for all $x \in \R^d$ and $m_x \in \MM$
$$
h(\C((x,m_x), \P)) = h(x + C^{\rho}_{\0}((x,m_x), \P)) = h(C^{\rho}_{\0}((x,m_x), \P)).
$$

Put $W_\la := [ -\frac{\la^{1/d}} {2}, \frac{\la^{1/d}} {2}]^d$ and $\hW_\la := W_\la \times \MM$, $\la>0$.
Given $h$ and a tessellation defined by the weight $\rho$, we define for all $\la > 0$
$$
H_{\lambda}^{\rho}(\Pm \cap \hW_\la )  := \sum_{\hx \in \Pm \cap \hW_\la}  \frac{  h(\C(\hx, \Pm)) }  { \Vol( W_\la \ominus \C(\hx, \Pm) ) } \,
\1 \{ \C(\hx, \Pm) \subseteq W_\la \}.
$$
Here, for sets $A$ and $B$, $A \ominus B := \{x \in \R^d: B+x \subseteq A\}$ denotes the {\em erosion} of $A$ by $B$.
The statistic $H_{\lambda}^{\rho}(\Pm \cap \hW_\la )$ disregards cells contained in the window $W_\la$ that are generated by the points outside $W_\la$.
Such cells do not exist in the Voronoi case but they could appear for weighted cells.
Therefore, we may also consider
$$
H_{\lambda}^{\rho}(\Pm)  := \sum_{\hx \in \Pm}  \frac{    h(\C(\hx, \Pm)) }  { \Vol( W_\la \ominus \C(\hx, \Pm) ) }\, \1\{ \C(\hx, \Pm) \subseteq W_\la \}.
$$

For every weight $\rho$ we define the score  $\xi^{\rho}: \hR \times {\bf N} \to \R$ by
\begin{equation}\label{repre}
\xi^{\rho}(\hx, \mathcal{A}) :=  h(\C(\hx, \mathcal{A}))\1\{\C(\hx,\mathcal{A}) \text{ is bounded}\}, \quad \hx \in \hR,\, \mathcal{A} \in {\bf N}.
\end{equation}
We use this representation to explicitly link our statistics with the stabilizing statistics in the literature \cite{BY05, BBY, LSY, Pe07, Pe, PY1, PY4}. Translation invariance for $h$ implies
\begin{equation*} 
\xi^{\rho}(\hx, \mathcal{A}) = \xi^{\rho}((x, m_x), \mathcal{A}) = \xi^{\rho}((\0, m_x),\mathcal{A}-x),
\end{equation*}
for every $\hx \in \hR,\, \hx := (x, m_x)$ and $\mathcal{A} \in {\bf N} $, where $\mathcal{A} - x := \{(a - x, m_a): (a, m_a) \in \mathcal{A}\}$.
If $\C(\hx,\Pm)$ is empty we put $\xi^{\rho}(\hx,\Pm) = h(\emptyset) = 0$. 
Write $\xi^{\rho}(\hx,\Pm) := \xi^{\rho}(\hx,\Pm \cup \{\hx\})$ for $\hx \not\in \Pm$.

\begin{defn}
The score $\xi^{\rho}$ is said to satisfy a $p$-moment condition, $p \in [1, \infty)$, if
\begin{equation} \label{mom}
\sup_{\hx, \hy \in \hR } \E | \xi^{\rho}(\hx, \Pm  \cup \{\hy\} )|^p < \infty.
\end{equation}
\end{defn}

For $r \in (0, \infty)$ and $y \in \R^d$, we denote by $B_r(y)$ the closed Euclidean ball of radius $r$ centered at $y$.
\begin{defn}\label{defD}
We say that the cells of the tessellation defined by $\rho$ and generated by $\P$ have diameters with exponentially decaying tails if
there is a constant $c_{diam} \in (0, \infty)$ such that
for all $\hx := (x,m_x)  \in \P$ there exists an almost surely finite random variable $D_{\hx}$ such that $\C(\hx,\Pm) \subseteq B_{D_{\hx}}(x)$ and
\begin{equation}\label{eq:stab}
\PP(D_{\hx} \ge t) \le c_{diam}\exp\left(-\frac{1}{c_{diam}}\, t^d \right), \quad t \geq 0.
\end{equation}
\end{defn}

\begin{defn}\label{defR}
We say that $\xi^{\rho}$ is stabilizing with respect to $\Pm$ if for all $\hx := (x,m_x) \in \Pm$ there exists an almost surely finite random variable $R_{\hx}:= R_{\hx}(\Pm)$, henceforth called a radius of
stabilization, such that
\begin{equation} \label{stabdef}
\xi^{\rho}(\hx,(\Pm \cup \mathcal{A}) \cap \hat{B}_{R_{\hx}}(x))  = \xi^{\rho}(\hx,\Pm \cup \mathcal{A})
\end{equation}
for all $\mathcal{A}$ with $\text{card}( \mathcal{A} ) \leq 7$ and
where $\hat{B}_r(y) := B_r(y) \times \MM$.  We say that $\xi^{\rho}$ is exponentially stabilizing with respect to $\Pm$ if there are constants $c_{stab}, \alpha \in (0, \infty)$ such that
\[
\PP(R_{\hx} \geq t) \leq c_{stab}\exp\left(-\frac{1}{c_{stab}}\,t^\alpha \right), \quad t \ge 0.
\]
\end{defn}

In other words,  $\xi^{\rho}$ is stabilizing with respect to $\cal P$ if there is $R_{\hx}$ such that the cell $\C(\hx,\Pm)$ is not affected by changes in point
configurations outside $\hat{B}_{R_{\hx}}(x)$.

Controlling the moments of $H_{\lambda}^{\rho}(\Pm \cap \hW_\la )$ is problematic since $\Vol( W_\la \ominus \C(\hx, \Pm) )$ may become arbitrarily small. It will therefore  be convenient to consider the following versions of $H_{\lambda}^{\rho}(\Pm \cap \hW_\la ) $ and $H_{\lambda}^{\rho}(\Pm)$.
Put
$$
\hat{H}_\la ^{\rho}(\Pm \cap \hW_\la ) := \sum_{\hx \in \Pm \cap \hW_\la} \frac{ h(\C(\hx, \Pm))\,\1 \{ \C(\hx,\Pm) \subseteq W_\la \}}  { \Vol( W_\la \ominus \C(\hx, \Pm) ) }
\, \1\{ \Vol(W_\la \ominus \C(\hx, \Pm) ) \geq \frac{\la} {2} \}
$$
and
$$
\hat{H}_\la ^{\rho}(\Pm) := \sum_{\hx \in \Pm} \frac{  h(\C(\hx, \Pm)) \, \1\{ \C(\hx,\Pm) \subseteq W_\la \}}  { \Vol( W_\la \ominus \C(\hx, \Pm) ) }
\, \1\{ \Vol(W_\la \ominus \C(\hx, \Pm) ) \geq  \frac{\la} {2} \}.
$$

By $\eta_\lambda, \lambda \in (0, \infty),$ we denote a homogeneous marked Poisson point process on $\hR$ such that the unmarked process on $\R^d$ has rate $\lambda$.
We write $\eta$ for $\eta_1$.
Our main results establish the limit theory for the above estimators and go as follows.
We assume the marks of $\Pm$ and $\eta$ belong to the interval $\MM := [0, \mu]$ for some constant $\mu \in [0, \infty)$.

\begin{thm}  \label{unbias}
Let $\Pm$ be an independently marked stationary point process with unit intensity and with marks following the law $\mathbb{Q}_{\mathbb{M}}$.
 Let $h: {\bf F}^d \rightarrow \R$ be a translation invariant function as above.
Let $M_\0$ be a random mark distributed according to $\Q_\MM$.

\begin{enumerate}[(i)]
\item The statistic $ H_{\lambda}^{\rho}(\Pm)$ is an unbiased estimator of $\E h(K^{\rho}_\0)$.

\item If $\xi^{\rho}$ satisfies the $p$-moment condition \eqref{mom} for some $p \in (1, \infty)$ and if the cell $C^{\rho}((\0,M_\0),\eta)$ has a diameter with an exponentially decaying tail, then  $H_{\lambda}^{\rho}(\eta \cap \hW_\la),$ $\hat{H}_{\lambda}^{\rho}(\eta)$ and $\hat{H}_{\lambda}^{\rho}(\eta \cap \hW_\la)$ are asymptotically unbiased estimators of $\E h(K^{\rho}_\0)$.

\item Under the conditions of $(ii)$ and assuming that $\xi^\rho$ stabilizes
with respect to $\eta$ as at \eqref{stabdef}, the statistics $ H_{\lambda}^{\rho}(\eta)$,  $ H_{\lambda}^{\rho}(\eta \cap \hW_\la),$ $ \hat{H}_{\lambda}^{\rho}(\eta)$ and $\hat{H}_{\lambda}^{\rho}(\eta \cap \hW_\la)$ are consistent estimators of $\E h(K^{\rho}_\0)$.
\end{enumerate}
\end{thm}
\vskip.3cm

Note that $ H_{\la}^{\rho}(\P \cap \hW_\la),$ $ \hat{H}_{\lambda}^{\rho}(\P)$ and $ \hat{H}_{\lambda}^{\rho}(\P \cap \hW_\la)$ are not unbiased. Under the assumptions of Theorem \ref{unbias}, one instead has
$$
 \E H_{\lambda}^{\rho}(\Pm \cap \hW_\la ) = \E \left(h(K^{\rho}_\0)\frac{\Vol(W_\la \cap (W_\la \ominus K^{\rho}_\0))}{\Vol(W_\la \ominus K^{\rho}_\0)}\right),
$$
$$
 \E \hat{H}_{\lambda}^{\rho}(\Pm \cap \hW_\la ) = \E \left(h(K^{\rho}_\0)\frac{\Vol(W_\la \cap (W_\la \ominus K^{\rho}_\0))}{\Vol(W_\la \ominus K^{\rho}_\0)}\,\1\{\Vol(W_\la \ominus K^{\rho} _\0) \geq \frac{\la} {2} \}\right),\\
$$
and
$$
 \E \hat{H}_{\lambda}^{\rho}(\Pm) = \E \left(h(K^{\rho}_\0) \1 \{\Vol(W_\la \ominus K^{\rho}_\0) \geq  \frac{\la} {2} \}\right).\\
$$
The general form of the bias is given by Theorem 1 of \cite{B99}.

\vskip.3cm

Given the score $\xi^{\rho}$  at \eqref{repre}, put
\begin{align}\label{defsig}
\sigma^2(\xi^{\rho}) & :=  \E (\xi^{\rho} (\0_M, \eta))^2 \\
& \ \ + \int_{\R^d}  \left[\E \xi^{\rho}(\0_M, \eta \cup \{x_M\}) \xi^{\rho}(x_M, \eta \cup \{\0_M\}) - \E \xi^{\rho}(\0_M, \eta)\,\E\xi^{\rho}(x_M, \eta)\right] \,\dint x, \nonumber
\end{align}
where $\0_M := (\0,M_\0)$, $x_M := (x,M_x)$, and $M_\0$ and $M_x$ are independent random marks distributed according to $\Q_\MM$. Note that $\E h(K^{\rho}_{\0}(\eta)) = \E \xi^{\rho} (\0_M, \eta)$.

\begin{thm} \label{Thm1}
Let $h$ be translation invariant and assume that $\xi^{\rho}$ is exponentially stabilizing with respect to $\eta$.
\begin{enumerate}[(i)]
\item If $\xi^{\rho}$ satisfies the $p$-moment condition \eqref{mom} for some $p \in (2, \infty)$, then
\begin{equation} \label{varlimit}
\lim_{\la \to \infty}  \la \Var \hat{H}_{\lambda}^{\rho}(\eta \cap \hW_\la )  = \lim_{\la \to \infty} \la  \Var \hat{H}_{\lambda}^{\rho}(\eta)
= \sigma^2(\xi^{\rho}) \in [0, \infty).
\end{equation}

\item If $\sigma^2(\xi^{\rho}) \in (0, \infty)$ and if the $p$-moment condition \eqref{mom} holds for some $p \in (4, \infty)$, then
$$
 \sqrt{ \la} \left(  H_{\lambda}^{\rho}(\eta \cap \hW_\la )   -  \E H_{\lambda}^{\rho}(\eta \cap \hW_\la )   \right)   \todl N(0, \sigma^2(\xi^{\rho}))
$$
and
$$
 \sqrt{ \la}  \left(  H_{\lambda}^{\rho}(\eta )   -  \E h(K^{\rho}_\0 (\eta) ) \right)   \todl N(0, \sigma^2(\xi^{\rho})),
$$
where $N(0, \sigma^2(\xi^{\rho}))$ denotes  a mean zero Gaussian random variable with  variance $ \sigma^2(\xi^{\rho})$.
\end{enumerate}
\end{thm}

\noindent{\em Remarks.} (i)  The assumption  $\sigma^2(\xi^{\rho}) \in (0, \infty)$ is often satisfied by scores of interest, as seen in the upcoming applications.
According to Theorem 2.1 in \cite{PY1}, where it has been shown that whenever we have
$$\frac{\sum_{\hx \in \eta \cap \hW_\la} (\xi^\rho (\hx, \eta) - \E \xi^\rho(\hx, \eta))}{\sqrt{\Var \sum_{\hx \in \eta \cap \hW_\la} \xi^\rho (\hx, \eta)}} \tod N(0, \sigma^2(\xi^\rho)),$$
then necessarily  $\sigma^2(\xi^{\rho}) \in (0, \infty)$ provided (a) there is a random variable $S < \infty$ and a random variable $\Delta^\rho(\infty)$ such that for all finite ${\cal A}
\subseteq \hat{B}_S(\0)^c$ we have
\begin{align*}
\Delta^\rho(\infty) = & \sum_{ \hx \in (\eta \cap \hat{B}_S(\0)) \cup {\cal A} \cup \{\0_M \} }
 \xi^\rho(\hx, (\eta \cap \hat{B}_S(\0)) \cup {\cal A} \cup \{\0_M\})\\
  &  \hspace{2cm} - \sum_{ \hx \in (\eta \cap \hat{B}_S(\0)) \cup {\cal A}  } \xi^\rho(\hx, (\eta \cap \hat{B}_S(\0)) \cup {\cal A} ),
 \end{align*}
and (b) $\Delta^\rho(\infty)$ is non-degenerate.  We will use this fact in showing positivity of $\sigma^2(\xi^{\rho})$ in the applications which follow.

\noindent(ii) Theorems \ref{unbias} and \ref{Thm1} hold for translation invariant statistics $h$ of Poisson--Voronoi cells regardless of the mark distribution 
because  $\xi^{\rho_1}$ stabilizes exponentially fast and diameters of Voronoi cells have exponentially decaying tails as  shown in \cite{Pe,PY1}.
In Section \ref{stabsection} we establish that the cells of the Laguerre and the Johnson--Mehl tessellations also have diameters with exponentially decaying tails and that $\xi^{\rho_i}, i=2, 3$ are exponentially stabilizing with respect to $\eta$.

\vskip.3cm

\noindent{\bf Applications.}   We provide some applications of our main results.  The proofs are provided in the sequel.  Our first result gives the limit theory for an unbiased estimator of the
distribution function of the volume of a typical cell in a generalized weighted Poisson--Voronoi tessellation.

\begin{thm} \label{thmApp1}
(i) For all $i= 1, 2, 3$ and  $t \in (0, \infty)$ we have that
\[
\sum_{\hx \in \etam}  \frac{ \1\{\Vol (C^{\rho_i}(\hx, \etam)) \leq  t \} }  { \Vol( W_\la \ominus C^{\rho_i}(\hx, \etam) ) }\,\1\{ C^{\rho_i}(\hx, \etam) \subseteq W_\la \}
\]
is an unbiased estimator of $\PP (\Vol(K^{\rho_i}_\0 (\eta) ) \leq  t)$.

\noindent (ii) It is the case that for all $ t \in (0, \infty)$
\begin{equation} \label{CLTapplic}
 \sqrt{ \la}  \left( \sum_{\hx \in \etam}  \frac{ \1\{\Vol (C^{\rho_i}(\hx, \etam)) \leq   t \} }  { \Vol( W_\la \ominus C^{\rho_i}(\hx, \etam) ) }\, \1\{ C^{\rho_i}(\hx, \etam) \subseteq W_\la \}   -  \PP (\Vol(K^{\rho_i}_\0 (\eta)  ) \leq  t ) \right)
\end{equation}
tends to  $N(0, \sigma^2(\varphi^{\rho_i}))$ in distribution  as $\lambda \to \infty$, where
$\varphi^{\rho_i}(\hx, \etam) := \1\{\Vol (C^{\rho_i}(\hx, \etam)) \leq  t \}$ and where $\sigma^2(\varphi^{\rho_i}) \in (0, \infty)$ is given by  \eqref{defsig}.
\end{thm}

Our next result gives the limit theory for an unbiased estimator of the $(d-1)$-dimensional Hausdorff measure $\H^{d-1}$  of the boundary of a typical cell in a generalized
weighted Poisson--Voronoi tessellation.

\begin{thm} \label{thmApp2}
(i) For all $i = 1,2,3$ we have that
$$ \sum_{\hx \in \etam  }  \frac{ {\cal H}^{d-1} ( \partial C^{\rho_i}(\hx, \etam)) }  { \Vol( W_\la \ominus C^{\rho_i}(\hx, \etam) ) }\,\1\{ C^{\rho_i}(\hx, \etam) \subseteq W_\la \}
$$
is an unbiased estimator of $\E {\cal H}^{d-1}(\partial K^{\rho_i}_\0 (\eta) ).$

\noindent (ii) It is the case that
$$
 \sqrt{ \la}  \left( \sum_{\hx \in \etam  }  \frac{ {\cal H}^{d-1} ( \partial C^{\rho_i}(\hx, \etam)) }  { \Vol( W_\la \ominus C^{\rho_i}(\hx, \etam) ) }\,\1\{ C^{\rho_i}(\hx, \etam) \subseteq W_\la \}   -  \E {\cal H}^{d-1}( \partial K^{\rho_i}_\0 (\eta) )  \right)
$$
tends to  $N(0, \sigma^2(\xi^{\rho_i}))$ in distribution as $\lambda \to \infty$, where \[\xi^{\rho_i}(\hx, \etam):= {\cal H}^{d-1} ( \partial C^{\rho_i}(\hx, \etam))\1\{C^{\rho_i}(\hx,\etam) \text{ is bounded}\}\] and where $\sigma^2(\xi^{\rho_i}) \in (0, \infty) $ is given by  \eqref{defsig}.
\end{thm}

There are naturally other applications of the general theorems.  By choosing $h$ appropriately, one could for example use the general results to deduce the limit theory for an unbiased estimator of the distribution function of
either the surface area, inradius, or circumradius of a typical cell in a generalized weighted Poisson--Voronoi tessellation.

\section{Stabilization of tessellations} \label{stabsection}
In this section we establish that (i) the cells in the Voronoi, Laguerre and Johnson--Mehl tessellations generated by Poisson input have diameters with exponentially decaying tails (see
Definition \ref{defD}) and (ii) the scores $\xi^{\rho_i}, \, i=1, 2, 3,$ as defined at \eqref{repre} are exponentially stabilizing (see Definition \ref{defR}). These two conditions arise in the
statements of  Theorems \ref{unbias} and \ref{Thm1}. Note that conditions (i) and (ii) have been already established in the case of the Poisson--Voronoi tessellation ($\rho_1$) in \cite{Pe} and \cite{PY1}. The Voronoi cell is a special example of both the Laguerre and the Johnson--Mehl cell when putting $\MM = \{0\}$ (or any constant). Thus it will be enough to show that these two conditions hold for the Laguerre ($\rho_2$) and the Johnson--Mehl ($\rho_3$) tessellations.

By definition we have
\[
\C(\hx,\Pm) = \bigcap_{\hz \in \Pm \setminus \{\hx\}} \mathbb{H}_{\hz}^{\rho}(\hx),
\]
where $\mathbb{H}_{\hz}^{\rho}(\hx) := \{y \in \R^d: \rho(y,\hx) \leq \rho(y,\hz)\}$.  Note that $\mathbb{H}_{\cdot}^{\rho}(\cdot)$ is a closed half-space in the context of the Voronoi and Laguerre
tessellations, whereas it has a hyperbolic boundary
for the Johnson--Mehl tessellation.
Tessellations generated by $\Pm$ are stationary and are examples of stationary particle processes, see \cite[Section 2.8]{BR} or \cite[Section 10.1]{SW08}.

\begin{prop} \label{lemm:radius} 
The cells of the tessellation defined by $\rho_i, i = 1,2,3,$ and generated by Poisson input $\eta$ have diameters with exponentially decaying tails as at \eqref{eq:stab}.
\end{prop}

\begin{proof} 
We need to prove \eqref{eq:stab} for all $\hx \in \eta$. Without loss of generality, we may assume that $\hx$ is the origin $\hat{\0} := (\0, m_\0)$ and we denote $D:=D_{\hat\0}$.

Let ${\cal K}_j$, $j=1,\dots,J$, be a collection of convex cones in $\R^d$ such that $\cup_{j=1}^J {\cal K}_j = \R^d$
and $\langle x,y\rangle \geq 3\|x\| \|y\| /4$ for any $x$ and $y$ from the same cone ${\cal K}_j$. Each cone has an apex at the origin $\0$. Denote $\hat{{\cal K}}_j:= {\cal K}_j \times \MM$.
We take $(x_j,m_j) \in \etam \cap \hat{\cal K}_j \cap \hat{B}_{2\mu}(\0)^c$ so that $x_j$ is closer to $\0$ than any other point from $\etam \cap \hat{\cal K}_j \cap
\hat{B}_{2\mu}(\0)^c$.
This condition means that the balls $B_{m_\0}(\0)$ and $B_{m_j}(x_j)$ do not overlap. Then
\[
C^{\rho_i}(\hat\0,\etam) \subseteq \bigcap_{j=1}^J \mathbb{H}_{(x_j,m_j)}^{\rho_i}(\hat\0), \quad i = 1,2,3.
\]
Therefore, it is sufficient to find $D$ such that for all $i = 1,2, 3$, we have $\mathbb{H}_{(x_j,m_j)}^{\rho_i}(\hat\0) \cap {\cal K}_j \subseteq B_D(\0)$ for $j=1,\dots,J$
to obtain $C^{\rho_i}(\hat\0,\etam) \subseteq B_D(\0)$.
Consider $y \in \mathbb{H}_{(x_j,m_j)}^{\rho_i}(\hat\0) \cap {\cal K}_j$. Then $\rho_i(y,\hat\0) \leq \rho_i(y,(x_j,m_j))$
and $\langle y,x_j\rangle \geq 3\|x_j\| \|y\| /4$.
For the Laguerre cell the first condition means that $\|y\|^2 - m_{\0}^2 \leq \|y - x_j\|^2 - m_j^2 = \|y\|^2 + \|x_j\|^2 - 2\langle y,x_j \rangle - m_j^2$. Thus
\[
2\langle y,x_j \rangle \le \|x_j\|^2+m_\0^2-m_j^2 \le \|x_j\|^2+\mu^2 < \frac32 \|x_j\|^2
\]
and so $\|y\| < \|x_j\|$.
For the Johnson--Mehl cell we have
\[
\|y-x_j\| \ge \|y\| - m_\0 + m_j \ge \|y\| - \mu,
\]
which for $\|y\| > \mu$ gives
\[
2\langle y,x_j\rangle \le 2\mu\|y\|-\mu^2+\|x_j\|^2.
\]
Hence, using the assumptions $\langle x_j,y\rangle \geq 3\|x_j\| \|y\| /4$ and $\|x_j\|>2\mu$,
\[
\|y\| \le \frac{2(\|x_j\|^2-\mu^2)}{3\|x_j\|-4\mu} < \frac{2\|x_j\|^2}{\|x_j\|} = 2\|x_j\|.
\]
Consequently, for either the Laguerre or Johnson--Mehl cells,  we can take
\begin{equation} \label{disD}
D=2\max_{j=1,\dots,J} \|x_j\|.
\end{equation}

Then, for $t \in (4\mu, \infty)$ we have
\begin{align*}
\PP(D \ge t) &\le \sum_{j=1}^J \PP(2\|x_j\| \ge t)
= \sum_{j=1}^J \PP(\eta \cap (\hat{B}_{t/2}(\0) \setminus \hat{B}_{2\mu}(\0)) \cap \hat{{\cal K}}_j = \emptyset) \\
&= \sum_{j=1}^J \exp(-\Vol((\hat{B}_{t/2}(\0) \setminus \hat{B}_{2\mu}(\0)) \cap \hat{{\cal K}}_j )) \leq  c_{diam} \exp\left(-\frac{1}{c_{diam}}\, t^d\right)
\end{align*}
for some $c_{diam} := c_{diam}(d, \mu) \in (0, \infty)$ depending on $d$ and $\mu$.
This shows Proposition \ref{lemm:radius} for $i = 2, 3$ and hence for $i = 1$ as well.    
\end{proof}

\begin{prop} \label{prop:stab} 
For all $i = 1, 2, 3$ the score $\xi^{\rho_i}$ defined at \eqref{repre} is exponentially stabilizing with respect to $\etam$.
\end{prop}

\begin{proof}
 We will prove \eqref{stabdef} when $\hx$ is the origin and we denote $R:=R_{\hat\0}$. For simplicity of exposition, we prove \eqref{stabdef} when $\mathcal{A}$ is the empty set, as the arguments do not change otherwise.
By \eqref{repre}, it is enough to show that there is an almost surely finite random variable $R$ such that
\begin{equation*}
C^{\rho_i}(\hat\0, \eta \cap \hat{B}_R(\0)) = C^{\rho_i}(\hat\0, (\eta \cap \hat{B}_R(\0)) \cup \{(z,m_z)\}) \quad \text{a.s.},
\end{equation*}
whenever $\|z\| \in (R, \infty)$.
To see this we put $R:= 2D + \mu$, where $D$ is at \eqref{disD}.
Given $\hat{z}:=(z,m_z)$, with $\|z\| \in (R, \infty)$, we assert that
$$
B_D(\0) \subseteq \mathbb{H}_{\hz}^{\rho_i}(\hat\0).$$
To prove this, we take any point $y \in B_D(\0)$ and show that
\begin{equation}\label{distance}
\rho_i(y, \hat{\0}) \leq \rho_i(y, \hat{z}), \quad i = 1,2,3.
\end{equation}
Note that $y \in B_D(\0)$ implies $\|y-z\| \in (D+\mu, \infty)$. The proof of \eqref{distance} is shown for the Laguerre and Johnson--Mehl cases individually. First, assume that $C^{\rho_2}(\hat\0, \eta)$ is the cell in the Laguerre tessellation.
Then
\[
\rho_2(y, \hat{\0}) = \|y\|^2-m_\0^2 \le D^2 < (D+\mu)^2-\mu^2 < \|y-z\|^2 - \mu^2 \le \|y-z\|^2 - m_z^2 = \rho_2(y, \hat{z}),
\]
showing that $y \in \mathbb{H}_{\hz}^{\rho_2}(\hat\0)$.
For the Johnson--Mehl case,
\[
\rho_3(y, \hat{\0})=\|y\|-m_\0 \le D = (D+\mu)-\mu < \|y-z\| - \mu \le \|y-z\| - m_z =\rho_3(y, \hat{z}),
\]
thus again $y \in \mathbb{H}_{\hz}^{\rho_3}(\hat\0)$, which shows our assertion.

The radius $D$ at \eqref{disD} has a tail decaying exponentially fast,  showing that
$R$ also has the same property. Consequently, for all $i = 1,2,3$, the score $\xi^{\rho_i}$ is exponentially stabilizing with respect to $\etam$.
\end{proof}

\vskip.3cm

\noindent{\em Remarks.}
(i) The assertion $C^{\rho_i}(\hat\0, \Pm) \subseteq B_D(\0)$ holds for a larger class of marked point processes. We only need that the unmarked point process has at least
one  point in each cone  ${\cal K}_j \cap B_{2\mu}(\0)^c$, $j=1,\dots,J$,  with probability $1$. Consequently, scores $\xi^{\rho_i}, i = 1,2,3,$ are stabilizing with respect to such marked point processes.

\noindent (ii) Proposition \ref{prop:stab} implies that the limit theory developed in  \cite{MGY,PY1,PY4} for the total edge length and related stabilizing functionals of the Poisson--Voronoi tessellation extends to Poisson tessellation models with weighted Voronoi cells. Thus Proposition \ref{prop:stab} provides expectation and variance asymptotics, as well as normal convergence, for such functionals of the Poisson tessellation.

\noindent (iii)
Aside from weighted Voronoi tessellations, Propositions \ref{lemm:radius} and \ref{prop:stab} hold also for the Delaunay triangulation. On the other hand, Proposition \ref{lemm:radius} holds for Poisson-line tessellation, but Proposition \ref{prop:stab} does not.

\section{Proofs of the main results}

\noindent{\bf Preliminary lemmas.} In this section, we omit in the notation the dependence on the weight  $\rho$ that defines the tessellation. For simplicity, we write
$$
H_{\lambda}(\eta \cap \hW_\la ):= H_{\lambda}^{\rho}(\eta \cap \hW_\la ), \quad H_{\lambda}(\eta ):= H_{\lambda}^{\rho}(\eta),
$$
as well as
$$\hat{H}_{\lambda}(\eta \cap \hW_\la ) := \hat{H}_{\lambda}^{\rho}(\eta \cap \hW_\la ), \quad \hat{H}_{\lambda}(\eta ):= \hat{H}_{\lambda}^{\rho}(\eta ).
$$
Let us start with some useful first order results.

\begin{lem}\label{L5.1}
Under the assumptions of Theorem \ref{unbias}(ii), we have
$$\lim_{\la \to \infty} \la\,\E \left| H_{\lambda}(\eta \cap \hW_\la ) -  \hat{H}_{\lambda}(\eta\cap \hW_\la) \right| = 0.$$
\end{lem}

\begin{proof}
We denote by $\hat\Q$ the product of the Lebesgue measure on $\R^d$ and $\Q_\MM$.
By the Slivnyak--Mecke theorem \cite[Corollary 3.2.3]{SW08} and stationarity,
\begin{align*}
&\E \left|  H_{\lambda}(\eta \cap \hW_\la ) -  \hat{H}_{\lambda}(\eta\cap \hW_\la) \right| \\
&\leq \E \sum_{\hx \in \eta \cap \hW_\la} \frac{  |h(C(\hx,\eta))|}{\Vol(W_\la \ominus C(\hx,\eta))}\,\1\{C(\hx,\eta) \subseteq W_\la\}\,
\1\{\Vol(W_\la \ominus C(\hx, \eta) ) < \frac{\la} {2} \} \\
&= \int_{\hW_\la} \E \left( \frac{  |h(C(\hx,\eta))|}{\Vol(W_\la \ominus C(\hx,\eta))}\,\1\{C(\hx,\eta) \subseteq W_\la\}\,
\1\{\Vol(W_\la \ominus C(\hx, \eta) ) <  \frac{\la} {2} \}\right) \,\hat\Q(\dint \hx) \\
&= \int_{W_\la}\int_{\MM} \E \Biggl( \frac{ |h(C((\0,m),\eta))|}{\Vol(W_\la \ominus C((\0,m),\eta))}\,\1\{x \in W_\la \ominus C((\0,m),\eta)\} \\
&\hspace{3cm}\times \1\{\Vol(W_\la \ominus C((\0,m), \eta) ) <  \frac{\la} {2} \}\Biggr) \,\Q_\MM(\dint m)\,\dint x.
\end{align*}
Changing the order of integration we get
\begin{align}\label{eq5.1}
\E \left|  H_{\lambda}(\eta \cap \hW_\la ) -  \hat{H}_{\lambda}(\eta\cap \hW_\la) \right| &\leq \int_{\MM} \E \bigg( |h(C(\0_m,\eta))|
\1\{ \Vol( W_\la \ominus  C(\0_m,\eta)) <  \frac{\la} {2} \} \nonumber\\
&\quad\times \int_{W_\la}
\frac{\1\{ x \in W_\la \ominus C(\0_m,\eta) \}}{\Vol(W_\la \ominus C(\0_m,\eta))} \,\dint x\bigg)\,\Q_\MM(\dint m),
\end{align}
where $\0_m := (\0,m)$.
The inner integral over $W_\la$  is bounded by one, showing that for all $p \in (1, \infty)$ we have
\begin{align*}
& \E \left|  H_{\lambda}(\eta \cap \hW_\la ) -  \hat{H}_{\lambda}(\eta\cap \hW_\la) \right|\\
&\quad\leq \int_{\MM} \E \left( |h(C((\0,m),\eta))|\,\1\{ \Vol( W_\la \ominus  C((\0,m),\eta)) <  \frac{\la} {2} \}\right)\,\Q_\MM(\dint m)  \\
 &\quad\leq  \int_{\MM} (\E |h(C((\0,m),\eta))|^{p})^{\frac{1}{p}} \, \left( \PP( \Vol( W_\la \ominus  C((\0,m),\eta)) <  \frac{\la} {2} \right)^{\frac{p-1}{p}} \,\Q_\MM(\dint m).
\end{align*}
The random variable $D$ at \eqref{disD}  satisfies $C(\hat\0,\etam) \subseteq B_D(\0)$ a.s. Thus,
$$
\PP\left( \Vol( W_\la \ominus  C(\hat\0,\etam) )  <  \frac{\la} {2} \right) \leq \PP\left( \Vol( W_\la \ominus  B_D(\0) ) <  \frac{\la} {2} \right).
$$
The volume of the erosion in the right hand side equals $(\la^{1/d}-2D)_+ ^d$. By conditioning on $Y := \1\{\la^{1/d}\geq 2D\}$, we obtain
\begin{align*}
\PP \left((\la^{1/d}-2D)_+ ^d <  \frac{\la} {2} \right) 
&= \PP\left((\la^{1/d}-2D)_+ ^d  <  \frac{\la} {2} | Y=1\right)\, \PP(Y=1) \\
&\quad + \PP\left((\la^{1/d}-2D)_+ ^d  <  \frac{\la} {2} |Y=0\right)\, \PP(Y=0)\\
& \leq \PP\left((\la^{1/d}-2D) ^d  <  \frac{\la} {2} \right) + \PP(\la^{1/d}< 2D)\\
& \leq 2 \PP (D > e(\la)),
\end{align*}
where $e(\la):=(\la^{1/d} - (\la/2)^{1/d})/2$. Finally, recalling  that $D$ has exponentially decaying tails as at \eqref{eq:stab}, we obtain
$$\PP\left( \Vol( W_\la \ominus  C(\hat\0,\etam) )  <  \frac{\la} {2} \right) \le 2 \, c_{diam} \exp\left(-\frac{1}{c_{diam}}\, e(\la)^d \right).$$

Using this bound we have
\begin{align*}
& \la\,\E \left|  H_{\lambda}(\eta \cap \hW_\la ) -  \hat{H}_{\lambda}(\eta\cap \hW_\la) \right|\\
& \leq  \la\int_{\MM} (\E |h(C((\0,m),\eta))|^{p})^{\frac{1}{p}} \left( 2\, c_{diam} \exp\left(- \frac{1}{c_{diam}}\, e(\lambda)^d\right)\right)^{\frac{p-1}{p}}\,\Q_\MM(\dint m).
\end{align*}
Now $\xi$ satisfies the $p$-moment condition for $p \in (1, \infty)$ and so Lemma \ref{L5.1} follows. 
\end{proof}

\begin{lem}\label{L5.2}
Under the assumptions of Theorem \ref{unbias}(ii), we have
$$\lim_{\la \to \infty} \la\,\E \left|  H_{\lambda}(\eta) -  \hat{H}_{\lambda}(\eta) \right| = 0.$$
\end{lem}

\begin{proof} We follow the proof of Lemma \ref{L5.1}. In \eqref{eq5.1}, we integrate over $\R^d$ instead of over $W_\la$,
yielding a value of one for the inner integral.  Now follow the proof of Lemma \ref{L5.1} verbatim.  
\end{proof}

\begin{lem}\label{L5.3}
Under the assumptions of Theorem \ref{unbias}(ii), we have
$$\lim_{\la \to \infty} \E \left|  \hat{H}_{\lambda}(\eta\cap \hW_\la) -  \hat{H}_{\lambda}(\eta) \right| = 0.$$
\end{lem}

\begin{proof} Write
\begin{align} \label{hnu}
\hnu_\la(\hx, \etam) & :=   \frac{ h(C( \hx,  \eta)) \,
 \1\{ C(\hx, \eta) \subseteq W_\la \}}  { \Vol( W_\la \ominus C(\hx,  \eta) ) }\, \\ \nonumber
&  \ \ \ \ \ \ \ \times  \1\{ \Vol( W_\la \ominus C( \hx, \eta) ) \geq  \frac{\la} {2}  \}
\,\1\{D_{\hx} \geq d(x, W_\la)\},
\end{align}
where $D_{\hx}$ is the radius of the ball centered at $x$ and containing $C(\hx, \etam)$
and where $D_{\hx}$ is equal in distribution to $D$, with $D$ at \eqref{disD}.  Here $d(x, W_\la)$ denotes the Euclidean distance between $x$ and $W_\la$.
We observe that
$$\E \left|  \hat{H}_{\lambda}(\eta\cap \hW_\la) -  \hat{H}_{\lambda}(\eta) \right| \leq \E  \sum_{\hx \in \eta \cap \hW_\la ^c} \left| \hnu_\la(\hx, \etam) \right|.$$

From now on, we use the notation $c$ to denote a universal positive constant whose value may change from line to line.
By the H\"older inequality, the $p$-moment condition on $\xi$, and Proposition \ref{lemm:radius} we have $\E |\hnu_\la(\hx, \etam)| \leq (c/\la)  \exp\left( - \frac{1}{c} d(x, W_\la)^d \right)$.
Thus
$$
\E \left|  \hat{H}_{\lambda}(\eta\cap \hW_\la) -  \hat{H}_{\lambda}(\eta) \right| \leq \frac{c}{\la}  \int_{W_\la^c} \exp\left( - \frac{1}{c} \, d(x, W_\la)^d \right)\,\dint x.
$$
Let $W_{\la,\eps}$ be the set of points in $W_\la^c$ at distance $\eps$ from $W_\la$.
The co-area formula implies
$$
\E \left|  \hat{H}_{\lambda}(\eta\cap \hW_\la) -  \hat{H}_{\lambda}(\eta) \right| \leq \frac{c}{\la} \int_0^{\infty} \int_{W_{\la,\eps}} \exp\left( - \frac{1}{c} \, \eps^d \right)\, {\cal H}^{d-1}(\dint y)\,\dint \eps.
$$
Since ${\cal H}^{d-1} ( W_{\la,\eps}) \leq c\,(\la^{1/d}(1 + \eps) )^{d-1}$, we get $\E \left|  \hat{H}_{\lambda}(\eta\cap \hW_\la) -  \hat{H}_{\lambda}(\eta) \right| = O(\la^{-1/d}).$
\end{proof}

\noindent{\bf Proof of Theorem \ref{unbias}.} (i) We have
\begin{align*}
\E H_{\lambda}(\Pm) &=  \E  \sum_{\hx \in \Pm}   \frac{  h(C(\hx, \Pm)) }  { \Vol( W_\la \ominus C(\hx, \Pm) ) }\,  \1\{ C(\hx,\Pm) \subseteq W_\la \}  \\
&=  \E  \sum_{\hx \in \Pm}  \frac{ h(C_{\0}(\hx, \Pm) ) }  { \Vol( W_\la \ominus C_{\0}(\hx, \Pm) ) } \, \1\{ x + C_{\0}(\hx,\Pm) \subseteq W_\la \}  \\
&= \int_{\R^d} \E   \left(   \frac{  h( K_{\0}^\rho  ) }  { \Vol( W_\la \ominus  K_{\0}^{\rho} ) }\,  \1\{ x + K_{\0}^{\rho}  \subseteq W_\la \}  \right)  \,\dint x \\
&=  \E \int_{\R^d}    \left(   \frac{  h( K_{\0}^\rho   ) }  { \Vol( W_\la \ominus  K_{\0}^\rho  ) }\,  \1\{ x \in   W_\la  \ominus K_{\0}^\rho \}  \right) \,\dint x \\
& = \E h(  K_{\0}^\rho ),
\end{align*}
where we use translation invariance of $h$, translation invariance of erosions, Campbell's theorem for stationary particle processes \cite[Theorem 2.41]{BR} or \cite[Section 4.1]{SW08}, and Fubini's theorem in this order.
Hence, we have shown the unbiasedness $H_{\lambda}(\Pm)$.

\vskip.3cm

\noindent (ii) The asymptotic unbiasedness of $H_{\lambda}(\eta \cap \hW_\la)$, $\hat{H}_{\lambda}(\eta \cap \hW_\la )$ and $\hat{H}_{\lambda}(\eta)$ is a consequence of Lemmas \ref{L5.1}, \ref{L5.2} and \ref{L5.3}. For example, concerning $H_{\lambda}(\eta \cap \hW_\la)$, one may write
\begin{align*}
& |\E H_{\lambda}(\eta \cap \hW_\la )-\E h(K_\0^{\rho} (\eta) )| \leq  \E |H_{\lambda}(\eta \cap \hW_\la ) - H_{\lambda}(\eta)| \\
& \leq  \left( \E |H_{\lambda}(\eta \cap \hW_\la ) - \hat{H}_{\lambda}(\eta \cap \hW_\la )| + \E |\hat{H}_{\lambda}(\eta \cap \hW_\la ) - \hat{H}_{\lambda}(\eta)| + \E |\hat{H}_{\lambda}(\eta) - H_{\lambda}(\eta)| \right),
\end{align*}
which in view of  Lemmas \ref{L5.1}, \ref{L5.2} and \ref{L5.3}
goes to zero as $\la \to \infty$.  This gives the asymptotic unbiasedness of $H_{\lambda}(\eta \cap \hW_\la )$. One may similarly show the asymptotic unbiasedness for $\hat{H}_{\lambda}(\eta \cap \hW_\la )$ and $\hat{H}_{\lambda}(\eta)$.

\vskip.3cm

\noindent (iii) To show consistency, we introduce  $T_\la(\eta \cap \hW_\la)= \la^{-1} \sum_{\hx \in \eta \cap \hW_\la} \xi(\hx, \eta)$. By assumption,  $\xi$ stabilizes and satisfies the $p$-moment
condition for $p \in (1, \infty)$. Thus, using Theorem 2.1 of \cite{PY4}, we get that $T_\la(\eta \cap \hW_\la)$ is a consistent estimator of $\E h(K_\0^\rho (\eta) ).$ To prove the consistency of the
estimators in Theorem \ref{unbias}(iii), it is enough to show for one of them that it has the same $L_1$ limit as $T_\la(\eta \cap \hW_\la)$. We choose $ \hat{H}_{\lambda}(\eta \cap \hW_\la )$ and write
\begin{align*}
&\E \left| \hat{H}_{\lambda}(\eta \cap \hW_\la )-T_\la(\eta \cap \hW_\la)\right| \\
&= \E \left|\la^{-1} \sum_{\hx \in \eta \cap \hW_\la} \xi(\hx, \eta) \left(  \frac{   \la\,\1\{ C(\hx, \eta) \subseteq W_\la \}\,\1\{  \Vol( W_\la \ominus C(\hx, \eta) ) \geq  \frac{\la}{2} \} }  { \Vol( W_\la \ominus C(\hx, \eta) ) } -1 \right) \right| \\
&\leq \la^{-1} \E  \sum_{\hx \in \eta \cap \hW_\la} |\xi(\hx, \eta)| \left| \frac{   \la\,\1\{ C(\hx, \eta) \subseteq W_\la \}\,\1\{  \Vol( W_\la \ominus C(\hx, \eta) ) \geq  \frac{\la}{2}
\} }  { \Vol( W_\la \ominus C(\hx, \eta) ) } -1 \right| \\
&\leq \int_{W_\la} \la^{-1} \E \left(   |h(K_\0^\rho  (\eta) )| \left| \frac{   \la\,\1\{ x + K_\0^\rho(\eta) \subseteq W_\la \}\,\1\{  \Vol( W_\la \ominus K_\0^\rho(\eta) ) \geq \frac{\la}{2} \}}  {
\Vol( W_\la \ominus K_\0^\rho (\eta) ) } -1 \right| \right)\,\dint x\\
&= \int_{[- \frac{1} {2}, \frac{1} {2 } ]^d} \E \big( |h(K_\0^\rho (\eta) )| Y_\la(u)\big)\,\dint u,
\end{align*}
where we substituted $\la^{1/d}u$ for $x$ in the last equality and defined random variables
$$
Y_{\la} (u) :=  \left| \frac{   \la\,\1\{ \la^{1/d} u + K_\0^\rho(\eta) \subseteq W_\la \}\,\1\{  \Vol( W_\la \ominus K_\0^\rho(\eta) ) \geq \frac{\la}{2} \}}  { \Vol( W_\la \ominus K_\0^{\rho}(\eta)}  ) -1 \right|.
$$

We show that $Y_{\la} (u)$ converges to zero in probability for any $u \in (-1/2, 1/2)^d$.  Using the inclusion $K_\0^{\rho}(\eta) \subseteq B_D(\0)$ given by  Proposition \ref{lemm:radius} and that $D$ has
exponentially decaying tails,  we conclude that both  $\la/\Vol(W_\la \ominus K_\0^\rho(\eta))$ and  $\1\{  \Vol( W_\la
\ominus K_\0^{\rho}(\eta) ) \geq \la/2  \}$ tend to one in probability. To prove the convergence of $Y_{\la}
(u)$ to zero in probability, it remains to show that $\1\{ \la^{1/d} u + K_\0^{\rho}(\eta) \subseteq W_\la \}$ converges to one in probability. Equivalently, we show that the probability of the event $\{\la^{1/d} u + K_\0^{\rho}(\eta) \subseteq W_\la \}$ goes to 1. Let $u \in (-1/2, 1/2)^d$ be fixed. Then
\begin{align*}
&\PP(\la^{1/d}u \in W_\la \ominus K_\0^\rho(\eta) ) \geq \PP(\la^{1/d}u \in W_\la \ominus B_D(\0)) \\
& = \PP \left(u \in \left[-\frac{1}{2} + \frac{D}{\la^{1/d}}, \frac{1}{2} - \frac{D}{\la^{1/d}} \right]^d \right)\\
& =  \PP \left(u \in \left[-\frac{1}{2} + \frac{D}{\la^{1/d}}, \frac{1}{2} - \frac{D}{\la^{1/d}} \right]^d | D\leq \log{\la}\right) \PP(D\leq \log{\la})\\
&\quad + \PP \left(u \in \left[-\frac{1}{2} + \frac{D}{\la^{1/d}}, \frac{1}{2} - \frac{D}{\la^{1/d}} \right]^d | D > \log{\la}\right) \PP(D > \log{\la}) \\
& \geq  \PP \left(u \in \left[-\frac{1}{2} + \frac{\log{\la}}{\la^{1/d}}, \frac{1}{2} - \frac{\log{\la}}{\la^{1/d}} \right]^d \right)\PP(D\leq \log{\la}) \\
&\quad + \PP \left(u \in \left[-\frac{1}{2} + \frac{D}{\la^{1/d}}, \frac{1}{2} - \frac{D}{\la^{1/d}} \right]^d | D > \log{\la}\right) \PP(D > \log{\la}).
\end{align*}
Again, $D$ has exponentially decaying tails, so the lower bound converges to $\PP (u \in (-1/2, 1/2)^d) = 1$, showing that $Y_\la(u)$ goes to zero in probability as $\la \to \infty$.
We proved that $Y_{\la} (u)$ converge to zero in probability, but they are also uniformly bounded by one, hence
it follows from the moment condition on $\xi$ that  $h(K_\0^\rho(\eta))Y_\la(u)$ goes to zero in $L^1$.
Finally, by the  dominated convergence theorem, we get
$$\lim_{\la \to \infty}\E \left| \hat{H}_{\lambda}(\eta \cap \hW_\la )-T_\la(\eta \cap \hW_\la)\right| = 0.$$
Thus $ \hat{H}_{\lambda}(\eta \cap \hW_\la )$ converges to $\E h(K_\0^\rho(\eta))$ in $L^1$ and also in probability. The consistency of the remaining estimators in Theorem \ref{unbias} follows  from Lemmas \ref{L5.1}, \ref{L5.2} and \ref{L5.3}.
This completes the proof of Theorem \ref{unbias}.  \qed

\vskip.5cm

\noindent{\bf Proof of Theorem \ref{Thm1} (i).}
We prove the variance asymptotics \eqref{varlimit}. The proof is split into two lemmas (Lemma \ref{lem:varhatTla} and Lemma
\ref{VarhatstarT}). We first show an auxiliary result used in the proofs of both lemmas. Then we prove the variance asymptotics for $\hat{H}_\la(\etam \cap \hW_\la)$. This is easier, since, after scaling by $\la$, the scores
are bounded by $2 |\xi(\hx,\eta)|$ and thus, by assumption, satisfy a $p$-moment condition for some $p \in (2, \infty)$. Finally, we conclude the proof by showing that the asymptotic variance of $\hat{H}_\la(\etam)$ is the same as the asymptotic variance of $\hat{H}_\la(\etam \cap \hW_\la)$.

\begin{lem}\label{L5}
Let $\varphi: \hR \times {\bf N} \to \R$ be an exponentially stabilizing function with respect to $\eta$ and which satisfies the $p$-moment condition for some $p \in (2, \infty)$. Then there exists a constant
 $c \in (0, \infty)$ such that for all $\hx, \hy \in \hR$
\begin{align} \label{assert}
&|\E \varphi(\hx, \etam \cup \{\hy\})\varphi(\hy, \etam \cup \{\hx\}) - \E \varphi(\hx, \etam)\,\E\varphi(\hy, \etam)|\\
&\qquad\leq {c}\,\left( \sup_{\hx, \hy \in \hR }  \E | \varphi(\hx, \etam  \cup \{ \hy \} )|^{p}\right)^{\frac{2}{p}}
\exp\left(-\frac{1}{c}\, \|x - y\|^{ \alpha }  \right).
\end{align}
\end{lem}

\begin{proof}
We follow the proof of Lemma 5.2 in \cite{BY05} and show that the constant $A_{1,1}$ there involves the moment $(\E | \varphi(\hx, \etam \cup \{\hy\})|^p)^{\frac{2}{p}}$. Put $R := \max (R_{\hx}, R_{\hat{y}})$, where $R_{\hx}, R_{\hat{y}}$ are the radii of stabilization as in Proposition \ref{prop:stab} for $\hx$ and $\hat{y}$, respectively. Furthermore, put $r:= \|x-y\|/3$ and
define the event $E := \{ R \leq  r \}$.  H\"older's inequality gives
\begin{align} \label{cov1}
& | \E \varphi(\hx, \etam \cup \{\hy\} ) \varphi(\hy, \etam \cup \{\hx\} ) - \E \varphi(\hx, \etam \cup \{\hy\} ) \varphi(\hy, \etam \cup \{\hx\} )\1\{E\} | \nonumber \\
&\qquad \leq c \left(\sup_{  \hx, \hy \in \hR    } \E | \varphi(\hx, \etam \cup \{\hy\} )|^p\right)^{\frac{2}{p}} \, \PP(E^c)^{\frac{p-2}{p}}.
\end{align}
Notice that
\begin{align*}
& \E \varphi(\hx, \etam \cup \{\hy\} ) \varphi(\hy, \etam \cup \{\hx\} )\1\{E\} \\
&\qquad = \E \varphi(\hx, (\etam \cup \{\hy\}) \cap \hat{B}_{r}(\hx)  ) \varphi(\hy, (\etam \cup \{\hx\}) \cap \hat{B}_{r}(\hx)  )\1\{E\} \\
&\qquad =
\E \varphi(\hx, (\etam \cup \{\hy\}) \cap \hat{B}_{r}(\hx)  ) \varphi(\hy, (\etam \cup \{\hx\}) \cap \hat{B}_{r}(\hx)  ) (1 - \1\{E^c\}).
\end{align*}

A second application of  H\"older's inequality gives
\begin{align} \label{cov2}
& | \E \varphi(\hx, \etam \cup \{\hy\} ) \varphi(\hy, \etam \cup \{\hx\} )\1\{E\}  
- \E \varphi(\hx, (\etam \cup \{\hy\})  \cap \hat{B}_{r}(\hx) ) \varphi(\hy, (\etam \cup \{\hx\})  \cap \hat{B}_{r}(\hy) )  | 
\nonumber \\
&\qquad \leq c \left(\sup_{  \hx, \hy \in \hR    } \E | \varphi(\hx, \etam \cup \{\hy\} )|^p\right)^{\frac{2}{p}} \, \PP(E^c)^{\frac{p-2}{p}}.
\end{align}
Thus, combining \eqref{cov1} and \eqref{cov2} and using independence of $\varphi(\hx, (\etam \cup \{\hy\})  \cap \hat{B}_{r}(\hx) )$ and
$\varphi(\hy, (\etam \cup \{\hx\})  \cap \hat{B}_{r}(\hy) )$ we have
\begin{align} \label{cov3}
& | \E \varphi(\hx, \etam \cup \{\hy\} ) \varphi(\hy, \etam \cup \{\hx\} )
- \E \varphi(\hx, (\etam \cup \{\hy\})  \cap \hat{B}_{r}(\hx) ) \E \varphi(\hy, (\etam \cup \{\hx\})  \cap \hat{B}_{r}(\hy) )  | \nonumber  \\
&\qquad \leq c \left(\sup_{  \hx, \hy \in \hR    } \E | \varphi(\hx, \etam \cup \{\hy\} )|^p\right)^{\frac{2}{p}} \, \PP(E^c)^{\frac{p-2}{p}}.
\end{align}
Likewise we may show
\begin{align} \label{cov4}
& | \E \varphi(\hx, \etam  ) \E \varphi(\hy, \etam  ) - \E \varphi(\hx, \etam  \cap \hat{B}_{r}(\hx) ) \E \varphi(\hy, \etam \cap \hat{B}_{r}(\hy) )  |  \nonumber \\
&\qquad \leq c \left(\sup_{  \hx, \hy \in \hR    } \E | \varphi(\hx, \etam \cup \{\hy\} )|^p\right)^{\frac{2}{p}} \, \PP(E^c)^{\frac{p-2}{p}}.
\end{align}
Combining \eqref{cov3} and \eqref{cov4} and using that $\PP(E^c)$ decreases exponentially in $\|x - y\|^{\alpha}$, we thus obtain \eqref{assert}.
\end{proof}

\begin{lem} \label{lem:varhatTla} If $\xi$ is exponentially stabilizing with respect to $\eta$ then 
\[
\lim_{\la \to \infty} \la \Var \hat{H}_\la(\etam \cap \hW_\la) = \sigma^2(\xi),
\]
where $\sigma^2(\xi)$ is at \eqref{defsig}.
\end{lem}

\begin{proof}
Put for all $\hx \in \hR$ and any marked point process $\Pm$,
$$
\zeta_\la(\hx, \Pm) := \frac{ \la\,\xi( \hx,  \Pm) }  { \Vol( W_\la \ominus C(\hx,  \Pm) ) }\, \1\{ \Vol( W_\la \ominus C( \hx, \Pm) ) \geq  \frac{\la} {2}  \}
$$
and
$$
\nu_\la(\hx, \Pm) :=  \zeta_\la(\hx,\Pm)\, \1\{ C(\hx, \Pm) \subseteq W_\la \}.
$$
Note that $\zeta_\la$ is translation invariant whereas $\nu_\la$ is not translation invariant.
Then $\la\,\hat{H}_\la(\etam \cap \hW_\la) = \sum_{\hx \in \etam \cap \hW_\la} \nu_\la(\hx,\etam)$.

Recall that $\hat\Q$ is the product measure of Lebesgue measure on $\R^d$ and  $\mathbb{Q}_{\MM}$. By the Slivnyak--Mecke theorem we have
\begin{align*}
& \la  \Var \hat{H}_\la(\etam \cap \hW_\la) = \la^{-1} \E \sum_{\hx \in \etam \cap \hW_\la} \nu_\la^2(\hx,\etam)   \\
&\qquad + \la^{-1} \E \sum_{\hx, \hy \in \etam \cap \hW_\la; \hx \neq \hy} \nu_\la (\hx,\etam)\nu_\la (\hy,\etam)- \la^{-1} \left(\E \sum_{\hx \in \etam \cap \hW_\la} \nu_\la(\hx,\etam)\right)^2 \\
& = \la^{-1} \int_{\hW_\la} \E \nu_\la^2(\hx, \etam)\,\hat\Q(\dint \hx) \\
&  \ \ \ \ \ + \la^{-1} \int_{\hW_\la}  \int_{\hW_\la} \left[ \E   \nu_\la(\hx, \etam \cup \{\hy\} )  \nu_\la(\hy, \etam \cup \{\hx\} ) - \E \nu_\la(\hx, \etam)\,\E \nu_\la (\hy, \etam)\right]
\,\hat\Q(\dint \hy)\,\hat\Q(\dint \hx) \\
& =: I_1(\la) + I_2(\la).
\end{align*}

Using stationarity and the transformation $u := \la^{1/d}x$ we  rewrite $I_1(\la)$ as
\begin{equation*}
I_1(\la) = \la^{-1} \int_{W_\la}\int_{\MM} \E Z_\la^2(\0_m,\etam,x)\,\Q_\MM(\dint m)\,\dint x = \int_{W_1} \E Z_\la^2(\0_M,\etam,\la^{1/d}u)\,\dint u,
\end{equation*}
where $Z_\la((z,m_z),\Pm,x) := \zeta_\la((z,m_z), \Pm)\,\1\{C((z,m_z),\Pm) \subseteq W_\la-x\}$.
Similarly, by translation invariance of $\zeta_\la$, we have
\begin{align*}
I_2(\la) &= \la^{-1} \int_{W_\la} \int_{W_\la-x} \int_{\MM} \int_{\MM} [ \E Z_\la(\0_{m_1}, \etam \cup \{z_{m_2}\}, x)\,Z_\la(z_{m_2}, \etam \cup \{\0_{m_1}\}, x) \\
&\hspace{2.75cm} - \E Z_\la(\0_{m_1}, \etam, x)\,\E Z_\la(z_{m_2}, \etam, x)]\,\Q_\MM(\dint m_1)\,\Q_\MM(\dint m_2)\,\dint z\,\dint x \\
&= \int_{W_1} \int_{W_\la-\la^{1/d}u} [\E Z_\la(\0_M,\etam \cup \{z_M\},\la^{1/d}u)\, Z_\la(z_M,\etam \cup \{\0_M\},\la^{1/d}u) \\
&\hspace{2.75cm} - \E Z_\la(\0_M,\etam,\la^{1/d}u)\,\E Z_\la(z_M,\etam,\la^{1/d}u)]\,\dint z\,\dint u,
\end{align*}
where $\0_{m_1} := (\0,m_1)$, $z_{m_2} := (z,m_2)$,
$\0_M := (\0,M_\0)$, $z_M := (z,M_z)$ and $M_\0$, $M_z$ are random marks distributed according to $\Q_\MM$.

Since $|\zeta_\la(\hx,\etam)| \leq 2|\xi(\hx,\etam)|$, $\zeta_\la$ satisfies a $p$-moment condition, $p \in (2, \infty)$. Recall that
$\Vol(W_\la \ominus C(\hx,\etam))/\la$ tends in probability to $1$ and notice that $W_\la - \la^{1/d}u$ for $u \in (-1/2,1/2)^d$ increases to $\R^d$ as $\la \to \infty$.
Thus, as $\la \to \infty$, we have for any $\hat\0 := (\0,m_\0)$, $\hat{z} := (z,m_z) \in \hR$ and $u \in (-1/2,1/2)^d$,
\begin{align}
\E Z_\la(\hat\0,\etam,\la^{1/d}u) &\to \E \xi(\hat\0,\etam), \label{L2} \\
\E Z_\la^2(\hat\0, \etam,\la^{1/d}u) &\to \E \xi^2(\hat\0, \etam), \label{L3} \\
\E Z_\la(\hat\0, \etam \cup \{ \hat{z}\}, \la^{1/d}u)Z_\la(\hat{z},\etam \cup \{ \hat\0 \} ,\la^{1/d}u) &
\to \E\xi(\hat\0, \etam  \cup \{ \hat{z}\}) \xi(\hat{z}, \etam \cup \{\hat\0 \}). \label{L4}
\end{align}

These ingredients are enough to establish variance asymptotics for $\hat{H}_\la(\etam \cap \hW_\la)$.
Indeed, $I_1(\la)$ converges to $\E \xi^2(\0_M, \etam)$ by \eqref{L3}.
Concerning $I_2(\la)$, for each $u \in (-1/2,1/2)^d$ we have
\begin{align*}
& \lim_{\la \to \infty}  \int_{W_\la-\la^{1/d}u}
[\E Z_\la(\0_M,\etam \cup \{z_M\},\la^{1/d}u) Z_\la(z_M,\etam \cup \{\0_M\},\la^{1/d}u) \\
&\hspace{3cm} - \E Z_\la(\0_M,\etam,\la^{1/d}u)\,\E Z_\la(z_M,\etam,\la^{1/d}u)]\,\dint z\\
&= \int_{\R^d} [ \E \xi(\0_M,\etam \cup \{z_M\}) \xi(z_M,\etam \cup \{\0_M\}) - \E \xi(\0_M,\etam) \E \xi(z_M,\etam)]\,\dint z.
\end{align*}
Here we use that for any $x \in \R^d$, the function $Z_\la(\cdot, \cdot, x): \hR \times {\bf N} \to \R$ is exponentially stabilizing with respect to $\eta$ and satisfies the $p$-moment condition for some $p \in (2, \infty)$ Thus, from Lemma \ref{L5}, the integrand is dominated by an exponentially decaying function of $\|z\|^{  \alpha } $. Applying the dominated convergence theorem, together with \eqref{L2} and \eqref{L4}, we obtain the desired variance asymptotics since $\Vol(W_1) = 1$.
\end{proof} 
\vskip.3cm
The next lemma completes the proof of Theorem \ref{Thm1} (i).

\begin{lem} \label{VarhatstarT} If $\xi$ is exponentially stabilizing with respect to $\eta$ then 
\[
\lim_{\la \to \infty} \la \Var \hat{H}_\la(\etam) = \lim_{\la \to \infty} \la \Var \hat{H}_\la(\etam\cap \hW_\la ) = \sigma^2(\xi).
\]
\end{lem}
\begin{proof}
Write
$$
\la\,\hat{H}_\la(\etam)  =  \sum_{\hx \in \etam \cap \hW_\la} \nu_\la(\hx,\etam) + \sum_{\hx \in \etam \cap \hW_\la^c} \nu_\la(\hx,\etam).
$$
Now
\begin{align*}
\la \Var \hat{H}_\la(\etam) &=  \la^{-1} \Var \left( \sum_{\hx \in \etam \cap \hW_\la} \nu_\la(\hx,\etam)
\right) +  \la^{-1} \Var \left(  \sum_{\hx \in \etam \cap \hW_\la^c} \nu_\la(\hx,\etam) \right)  \\
&\quad \ \ \ \ + 2 \la^{-1} \Cov \left( \sum_{\hx \in \etam \cap \hW_\la} \nu_\la(\hx,\etam),  \sum_{\hx \in \etam \cap \hW_\la^c} \nu_\la(\hx,\etam) \right).
\end{align*}
It suffices to show $\Var \left(  \sum_{\hx \in \etam \cap \hW_\la^c} \nu_\la(\hx,\etam) \right) = O(\la^{(d-1)/d})$,
for then the Cauchy--Schwarz inequality shows that the covariance term in the above expression is negligible compared to $\la$.

Now we show $\Var \left(  \sum_{\hx \in \etam \cap \hW_\la^c} \nu_\la(\hx,\etam) \right) = O(\la^{(d-1)/d})$ as follows.
Note that $\hat{H}_\la(\etam) = \sum_{\hx \in \etam} \hat{\nu}_\la(\hx, \etam),$ where $\hat{\nu}_\la(\hx, \etam)$ is at \eqref{hnu}.
By the Slivnyak--Mecke theorem we have
\begin{align*}
& \la \Var  \left(  \sum_{\hx \in \etam \cap \hW_\la^c} \nu_\la(\hx,\etam) \right) = \la^{-1} \E \sum_{\hx \in \etam \cap \hW_\la^c} \hnu_\la^2(\hx,\etam)  \\
&\qquad + \la^{-1} \E \sum_{\hx, \hy \in \etam \cap \hW_\la^c; \hx \neq \hy} \hnu_\la (\hx,\etam)\hnu_\la (\hy,\etam) - \la^{-1} \left(\E \sum_{\hx \in \etam \cap \hW_\la^c} \hnu_\la(\hx,\etam)\right)^2 \\
& = \la^{-1} \int_{\hW_\la^c} \E \hnu_\la^2(\hx, \etam)\,\hat\Q(\dint \hx) \\
&\qquad + \la^{-1} \int_{\hW_\la^c}  \int_{\hW_\la^c} [ \E   \hnu_\la
(\hx, \etam \cup \{\hy\} )  \hnu_\la(\hy, \etam \cup \{\hx\} ) - \E \hnu_\la(\hx, \etam)\, \E \hnu_\la (\hy, \etam)]\,\hat\Q(\dint \hx)\,\hat\Q(\dint \hy)\\
& =: I_1^*(\la) + I_2^*(\la).
\end{align*}

By the H\"older inequality, the moment condition on $\xi$ and Proposition \ref{lemm:radius} we have $\E \hnu_\la(\hx, \etam)^p \leq c \exp\left( - \frac{1}{c} \, d(x, W_\la)^d \right)$
for some positive constant $c$. Then, similarly as in Lemma \ref{L5.3}, we may use the co-area formula to obtain $I_1^*(\la) = O(\la^{-1/d}).$

To bound $I_2^*(\la)$ we appeal to Lemma \ref{L5}. Notice that $|\hnu_\la(\hx,\etam)| \leq 2|\xi(\hx,\etam)|$.
Since $\hnu_\la, \la \geq 1,$ are exponentially stabilizing with respect to  $\etam$ and satisfy the $p$-moment condition for $p \in (2, \infty)$, then
by Lemma \ref{L5}
\begin{align*}
& | \E   \hnu_\la
(\hx, \etam \cup \{\hy\} )  \hnu_\la(\hy, \etam \cup \{\hx\} ) - \E \hnu_\la(\hx, \etam)\, \E \hnu_\la (\hy, \etam)|\\
&\qquad  \leq c\,\left(\sup_{  \hx, \hy  \in \hR    } \E | \hnu_\la(\hx,  \etam \cup \{\hy\}  )|^p\right)^{\frac{2}{p}} \exp\left(-\frac{1}{c} \, \|x - y\|^\alpha \right).
\end{align*}
Using this estimate we compute
\begin{align*}
I_2^*(\la) &\leq \la^{-1} \int_{\hW_\la^c} \int_{W_\la^c} c \,(\E | \hnu_\la(\hx, \etam)|^{p})^{\frac{2}{p}} \exp\left(-\frac{1}{c}\, \|x - y\|^\alpha \right)\,\dint y\,\hat\Q(\dint \hx) \\
&\leq c\, \la^{-1} \int_{\hW_\la^c} ( \E | \hnu_\la(\hx, \etam)|^{p})^{\frac{2}{p}} \int_{\R^d} \exp\left(-\frac{1}{c} \, \|x - y\|^\alpha\right)\,\dint y\,\hat\Q(\dint \hx) \\
&\leq c\, \la^{-1} \int_{W_\la^c} \exp\left( -\frac{1}{c} \, d(x, W_\la)^d \right) \,\dint x  \int_{\R^d} \exp\left(-\frac{1}{c}\, \|y\|^\alpha\right)\,\dint y.
\end{align*}
Since  $\int_{\R^d}   \exp(- \| y\|^\alpha/c )\,\dint y < \infty$, we obtain
$$
I_2^*(\la)  \leq   c\, \la^{-1} \int_{W_\la^c}   \exp\left( -\frac{1}{c} \, d(x, W_\la)^d \right)\,\dint x.
$$
Arguing as we did for $I_1^*(\la)$ we obtain $I_2^*(\la) = O(\la^{-1/d}).$
\end{proof}

\vskip.3cm
\noindent{\bf Proof of Theorem \ref{Thm1} (ii).} Now we prove the central limit theorems for $H_\la(\etam \cap \hW_\la)$ and $H_\la(\etam)$.
Let us first introduce some notation. Define for any stationary marked point process $\Pm$ on $\hR$, 
\begin{align*}
\xi_\la (\hx, \Pm) &: = \frac{\la\,\xi( \la^{1/d} \hx, \la^{1/d} \Pm) }  { \Vol( W_\la \ominus C( \la^{1/d} \hx, \la^{1/d} \Pm) ) } \, \1\{ C(\la^{1/d} \hx, \la^{1/d}
\Pm) \subseteq W_\la \},\\
\hat{\xi}_\la(\hx, \Pm) &:= \xi_\la(\hx, \Pm)\,\1\{ \Vol( W_\la \ominus C( \la^{1/d} \hx, \la^{1/d} \Pm) ) \geq  \frac{\la} {2}  \},
\end{align*}
where $\lambda^{1/d}\hx := (\lambda^{1/d} x,m_x)$ and $\lambda^{1/d}\Pm := \{\lambda^{1/d}\hx: \hx \in \Pm\}$.

Put
\begin{align*}
S_\la(\eta_\la \cap \hW_1) := \sum_{\hx \in \eta_\la \cap \hW_1} \xi_\la(\hx, \eta_\la), &\ \quad \hat{S}_\la(\eta_\la \cap \hW_1) := \sum_{\hx \in \eta_\la \cap \hW_1} \hat{\xi}_\la(\hx, \eta_\la),
\end{align*}
as well as
\begin{align*}
 S_\la(\eta_\la) := \sum_{\hx \in \eta_\la} \xi_\la(\hx, \eta_\la), &\ \quad \hat{S}_\la (\eta_\la) := \sum_{\hx \in \eta_\la} \hat{\xi}_\la(\hx, \eta_\la).
\end{align*}
Notice that
$$
S_\la(\eta_\la \cap \hW_1) \eqd \la\,H_\la(\eta \cap \hW_\la), \quad  S_\la(\eta_\la) \eqd \la\,H_\la(\eta) $$
and
$$
\hat{S}_\la(\eta_\la \cap \hW_1) \eqd \la\,\hat{H}_\la(\eta \cap \hW_\la) \quad \text{and}  \quad \hat{S}_\la(\eta_\la) \eqd \la\,\hat{H}_\la(\eta)$$
due to the distributional identity $\la^{1/d} \eta_\la \eqd \eta_1$.
The reason for  expressing the statistic $\la\,H_\la(\etam \cap \hW_\la)$ in terms of the scores $\xi_\la(\hx, \eta_\la)$  is that it puts us in a better position to apply the  normal approximation results of \cite{LSY} to the sums $S_\la(\eta_\la \cap \hW_1)$.

In particular we appeal to  Theorem 2.3 of \cite{LSY}, with $s$ replaced by $\la$ there,  to establish a central limit theorem for $\hat{S}_\la (\eta_\la \cap \hW_1)$.  Indeed, in that paper we may put $\XX$ to be $\R^d$, we let $\Q$ be Lebesgue measure on $\R^d$ so that $\eta_\la $ has intensity measure $\la \Q$, and we put $K = W_1$.
We may write $\hat{S}_\la(\eta_\la \cap \hat{W}_1) = \sum_{\hx \in \eta_\la \cap \hat{W}_1 }  \hat{\xi}_\la(\hx, \eta_\la)\,\1\{x \in W_1\}.$
 Note that $\hat{\xi}_\la(\hx, \eta_\la)\1\{x \in W_1\}, \hx \in \hat\XX,$ are exponentially stabilizing with respect to the input $\eta_\la$, they satisfy the $p$-moment condition for some $p \in (4, \infty)$,  they vanish for $x \in W_1^c$, and they (trivially) decay exponentially fast with respect to the distance to $K$. (Here the notion of decaying exponentially fast with respect to the distance to $K$ is defined at (2.8)  of \cite{LSY};  since the distance to $K$ is zero for $x \in K$ this condition is trivially satisfied.)
 This makes $I_{K, \la} = \Theta(\la)$ where $I_{K, \la}$ is defined at (2.10) of \cite{LSY}.  Thus all conditions of Theorem 2.3 of \cite{LSY} are fulfilled and
 we deduce a central limit theorem for $\hat{S}_\la(\eta_\la \cap \hW_1)$ and hence for $\hat{H}_\la(\eta \cap \hW_\la)$.

We may also apply Theorem 2.3 of \cite{LSY} to show a central limit theorem for $\hat{S}_\la(\eta_\la)$. For $x \in W_1^c$ we find the radius $D_x$ such that $C(\la^{1/d}\hx,\la^{1/d}\eta_\la) \subseteq B_{D_x}(\la^{1/d}x)$.
Then the score $\hat{\xi}_\la(\hx, \eta_\la)$ vanishes if $D_x > d(\la^{1/d}x,W_\la)$. As in Section \ref{stabsection}, $D_x$ has exponentially decaying tails and thus $\hat{\xi}_\la$ decays exponentially fast with respect to the distance to $K$.

Let $d_K(X, Y)$ denote the Kolmogorov distance between random variables $X$ and $Y$. Applying Theorem 2.3 of \cite{LSY} we obtain
$$
  d_K \left( \frac{ \hat{S}_\la (\eta_\la \cap \hW_1)   - \E \hat{S}_\la(\eta_\la \cap \hW_1) } { \sqrt{   \Var \hat{S}_\la (\eta_\la \cap \hW_1)    }},  N(0,1) \right) \leq  \frac{ c }  { \sqrt{   \Var \hat{S}_\la  (\eta_\la \cap \hW_1) }}
$$
and
$$
  d_K \left( \frac{ \hat{S}_\la (\eta_\la)   - \E \hat{S}_\la(\eta_\la) } { \sqrt{   \Var \hat{S}_\la(\eta_\la)    }},  N(0,1) \right) \leq  \frac{ c }  { \sqrt{   \Var \hat{S}_\la (\eta_\la)   }}.
$$
 Combining this with  \eqref{varlimit} and using $\Var \hat{S}_\la (\eta_\la \cap \hW_1)  \geq c\, \la$, we obtain as $\la \to \infty$
\[
  \frac{ \hat{S}_\la (\eta_\la \cap \hW_1)   - \E \hat{S}_\la (\eta_\la \cap \hW_1)} { \sqrt{ \la}}     \tod N(0, \sigma^2(\xi))
\]
and 
\[
\frac{ \hat{S}_\la (\eta_\la)   - \E \hat{S}_\la (\eta_\la)} { \sqrt{ \la}}     \tod N(0, \sigma^2(\xi)).
\]

To show that
\begin{equation} \label{CLT2}
  \frac{ {S}_\la (\eta_\la \cap \hW_1)  - \E {S}_\la(\eta_\la \cap \hW_1) } { \sqrt{ \la}}
  \tod N(0, \sigma^2(\xi)),
\end{equation}
as $\la \to \infty$, it suffices to show $\lim_{\la \to \infty} \E| S_\la (\eta_\la \cap \hW_1)- \hat{S}_\la(\eta_\la \cap \hW_1)| = 0$. Since $\E| S_\la (\eta_\la \cap \hW_1) - \hat{S}_\la (\eta_\la
\cap \hW_1)| = \la\,\E |H _\la (\eta_\la \cap \hW_\la) - \hat{H} _\la (\eta_\la \cap \hW_\la)|$, we may use Lemma \ref{L5.1} to prove \eqref{CLT2}.
 Likewise, to obtain the central limit theorem for $ {S}_\la (\eta_\la)$, it suffices to show $\lim_{\la \to \infty} \E| S_\la (\eta_\la)- \hat{S}_\la(\eta_\la)| = 0$, which is a consequence of Lemma \ref{L5.2}. Hence we deduce from the central limit theorem for $\hat{S}_\la(\eta_\la)$ that as $\la \to \infty$
\[
  \frac{ {S}_\la(\eta_\la) - \E {S}_\la(\eta_\la) } { \sqrt{ \la}} \eqd \sqrt{\la} \left( H_\la(\eta) - \E h(K_\0^\rho(\eta)) \right) \tod N(0, \sigma^2(\xi)).
\]
This completes the proof of Theorem \ref{Thm1} (ii).
 \qed

\section{Proofs of Theorems \ref{thmApp1} and \ref{thmApp2}}

Before giving the proof of Theorem \ref{thmApp1} we recall from Section \ref{stabsection} that 
translation invariant cell characteristics $\xi^{\rho_i}$ are exponentially stabilizing with respect to Poisson input $\eta$. This allows us to apply Theorem \ref{Thm1} to cell
characteristics of tessellations defined by $\rho_i, \, i=1, 2, 3$. For example, we can take $h(\cdot)$ to be either the volume or surface area of a cell or the  radius of the circumscribed or inscribed ball.

\vskip.3cm

\noindent{\bf Proof of Theorem \ref{thmApp1}.} (i)  The assertion of unbiasedness follows from Theorem \ref{unbias}(i).  (ii) To prove the asymptotic normality, we write
$$
h(C^{\rho_i}(\hx, \etam))  := \1\{\Vol (C^{\rho_i}(\hx, \etam)) \leq t \} =: \varphi^{\rho_i}(\hx, \etam).
$$
To deduce \eqref{CLTapplic} from Theorem \ref{Thm1}(ii) we need only verify the $p$-moment condition for $p \in (4, \infty)$ and the positivity of $\sigma^2(\varphi^{\rho_i})$. The moment
condition holds for all $p \ \in  [1, \infty) $ since $\varphi$ is bounded by $1$. To verify  the positivity of $\sigma^2(\varphi^{\rho_i})$, we recall Remark (i) following Theorem \ref{Thm1}.  More precisely we may
 use Theorem 2.1 of \cite{PY1}  and show that there is an a.s. finite random variable $S$ and a non-degenerate random variable $\Delta^{\rho_i}(\infty)$ such that for all finite  ${\cal A} \subseteq \hat{B}_S(\0)^c$ we have
\begin{align*}
\Delta^{\rho_i}(\infty) &  = \sum_{ \hx \in (\eta \cap \hat{B}_S(\0)) \cup {\cal A} \cup \{\0_M \} }
\1\{\Vol (C^{\rho_i}(\hx, (\eta \cap \hat{B}_S(\0)) \cup {\cal A} \cup \{\0_M\} ))\leq t \} \\
 &   \hspace{2.5cm}   - \sum_{ \hx \in (\eta \cap \hat{B}_S(\0)) \cup {\cal A} } \1\{\Vol (C^{\rho_i}(\hx, (\eta \cap \hat{B}_S(\0)) \cup {\cal A} )) \leq t \}.
\end{align*}
We first explain the argument for the Voronoi case and then indicate  how to extend it to treat the Laguerre and Johnson--Mehl tessellations.

Let $t \in (0,\infty)$ be arbitrary but fixed. Let $N$ be the smallest integer of even parity that is larger than $4\sqrt{d}$.
The choice of this value will be explained later in the proof.
For $L>0$ we consider a collection of $N^d$ cubes $Q_{L,1},\ldots, Q_{L,N^d}$ centered around  $x_i, i = 1,\ldots, N^d,$ such that
\begin{enumerate}
\item[(i)]  $Q_{L,i}$ has  side length $\frac{L} {N}$, and
\item[(ii)]  $\cup \{ Q_{L,i}, i=1, \ldots, N^d\} = [- \frac{L} {2} , \frac{L} {2} ]^d$.
\end{enumerate}
Put $\eps_L:= L/100N$  and $\hat{Q}_{L,i} := Q_{L,i} \times \MM$. Define the event
$$
E_{L, N} := \left\{ |\eta \cap \hat{Q}_{L,i} \cap \hat{B}_{\eps_L}(x_i)|=1,\, |\eta \cap \hat{Q}_{L,i} \cap \hat{B}_{\eps_L}^c(x_i)|=0, \forall
i=1,\ldots,N^d\right\}.
$$
Elementary properties of the Poisson point process show that $\PP(E_{L,N}) > 0$ for all $L$ and $N.$

On $E_{L,N}$ the faces of the tessellation restricted to $[-L/2, L/2]^d$ nearly coincide with the union of the boundaries of
$Q_{L,i}, i=1, \ldots, N^d$ and the cell generated by $\hx \in \eta \cap [-L/2+L/N,L/2-L/N]^d$ is determined only by $\eta \cap
(\cup \{Q_{L,j}, j \in I(\hx)\})$, where $j \in I(\hx)$ if and only if $\hx \in \hat{Q}_{L,j}$ or $\hat{Q}_{L,j} \cap \hat{Q}_{L,i}
\neq \emptyset$ for $i$ such that $\hx \in \hat{Q}_{L,i}$.
Thus inserting a point at the origin will not affect the cells far from the origin.
More precisely, the cells around the points outside ${\hat R}_{L,N} := [-2L/N,2L/N]^d \times \MM$ are not affected
by inserting a point at the origin. For $S_L := L/2$ we have ${\hat R}_{L,N} \subseteq {\hat B}_{S_L}(\0)$
due to our choice of the value $N$.
Therefore,
\[
C^{\rho_1}(\hx,(\eta \cap \hat{B}_{S_L}(\0)) \cup {\cal A} \cup \{\0_M\}) =
C^{\rho_1}(\hx,(\eta \cap \hat{B}_{S_L}(\0)) \cup {\cal A})
\]
for any finite ${\cal A} \subseteq {\hat B}_{S_L}(\0)^c$ and $\hx \in (\eta \cap ({\hat B}_{S_L}(\0) \setminus {\hat R}_{L,N})) \cup {\cal A}$.
Consequently, on $E_{L,N}$,
\begin{align*}
\Delta^{\rho_1}(\infty) &  = \sum_{ \hx \in (\eta \cap {\hat R}_{L,N}) \cup \{\0_M \} }
\1\{\Vol (C^{\rho_1}(\hx, (\eta \cap \hat{B}_{S_L}(\0)) \cup {\cal A} \cup \{\0_M\} ))\leq t \} \\
 &   \hspace{2.5cm}   - \sum_{ \hx \in \eta \cap {\hat R}_{L,N} } \1\{\Vol (C^{\rho_1}(\hx, (\eta \cap \hat{B}_{S_L}(\0)) \cup {\cal A} )) \leq t \}.
\end{align*}
Figure \ref{fig1} illustrates the difference appearing in $\Delta^{\rho_1}(\infty)$ on $E_{L,N}$ for $d=2$.
The ball $B_{S_L}(\0)$ is shown in blue whereas the square $[-2L/N,2L/N]^2$ is in red.
The cells generated by the points outside the red square are identical for both point configurations whereas the cells generated
by the points inside the red square may differ.

\begin{figure}
\begin{center}
\includegraphics[width=0.475\textwidth]{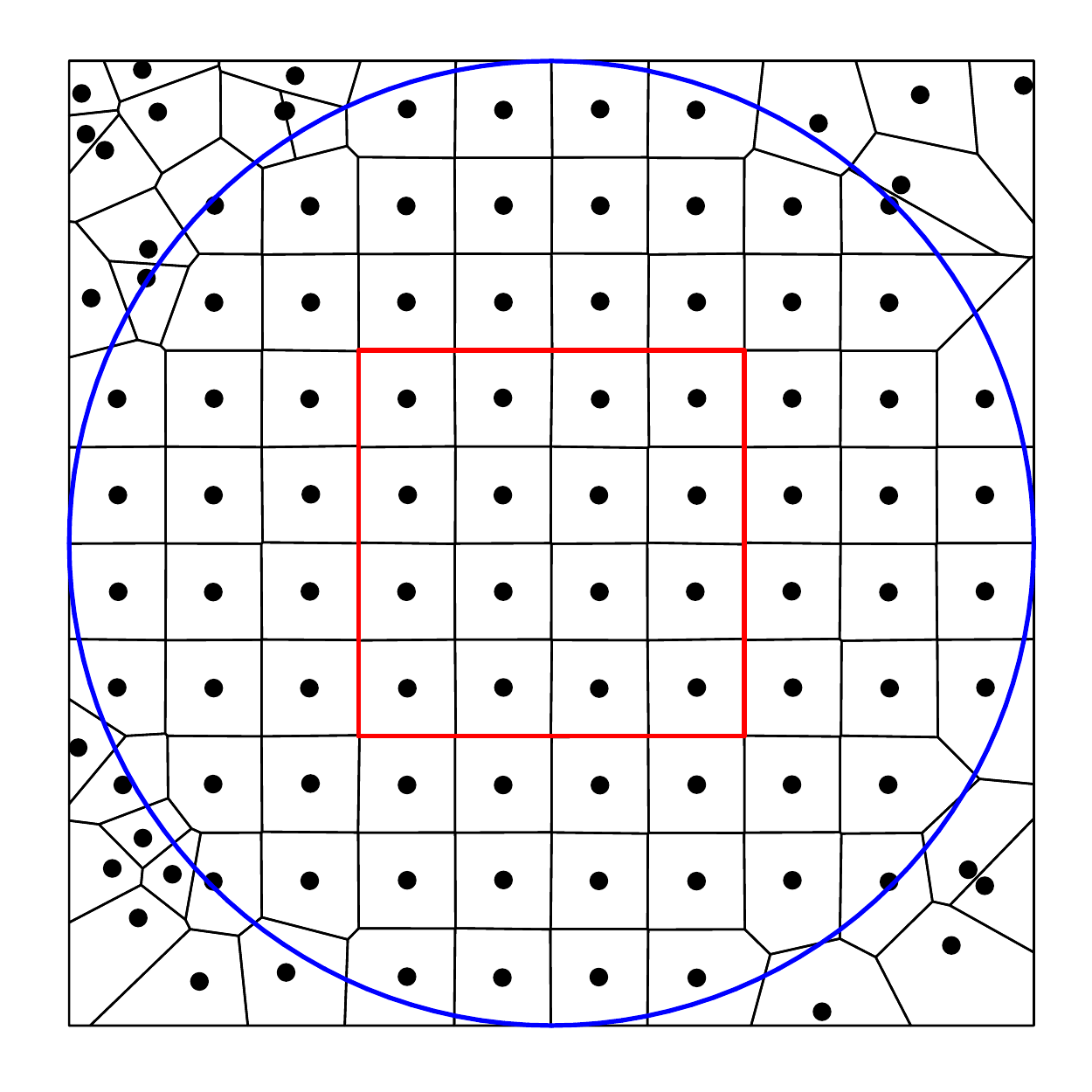}
\includegraphics[width=0.475\textwidth]{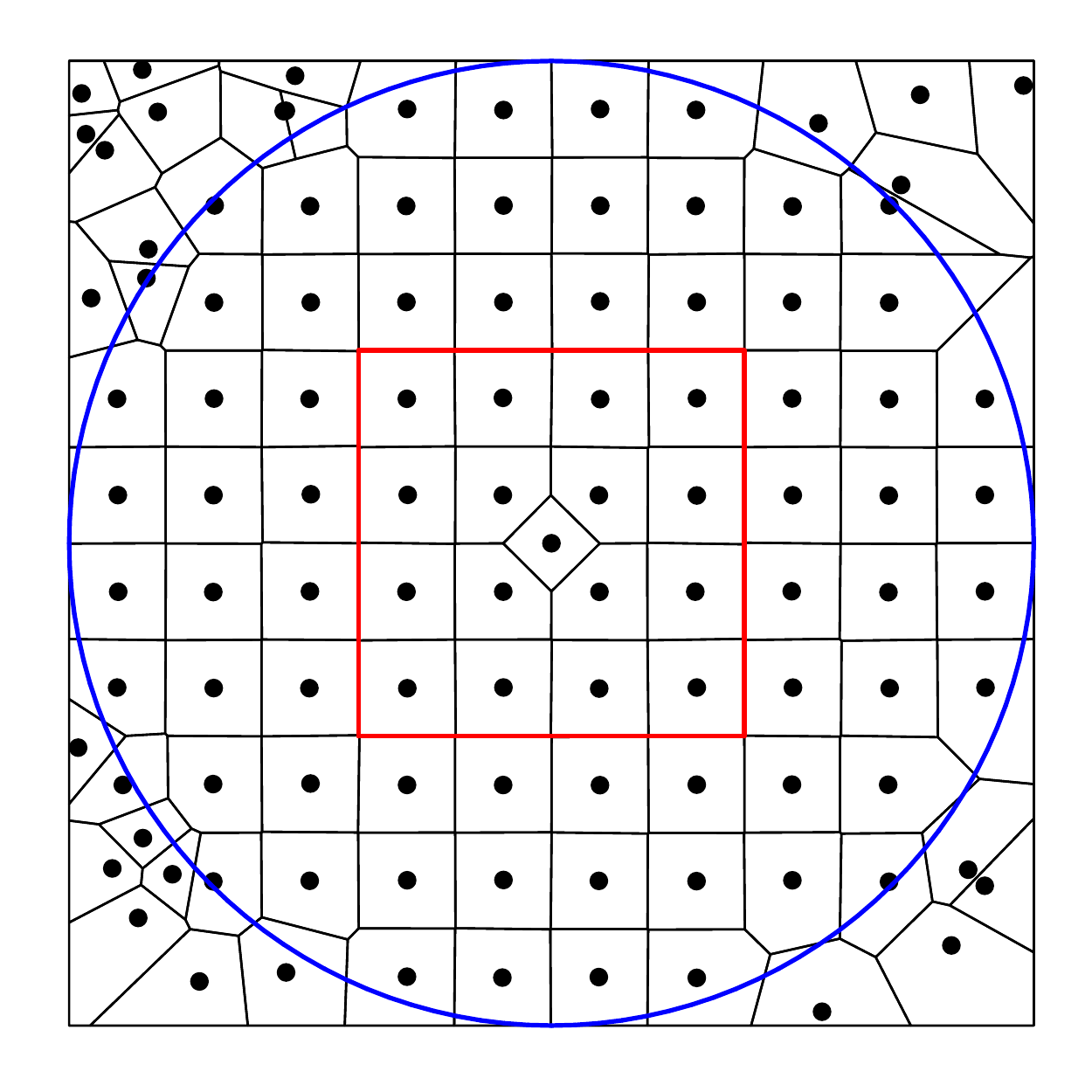}
\caption{Voronoi tessellations in $[-L/2,L/2]^2$ generated by $(\eta \cap \hat{B}_{S_L}(\0)) \cup {\cal A}$ (left)
and $(\eta \cap \hat{B}_{S_L}(\0)) \cup {\cal A} \cup \{\0_M\}$ (right).
The ball $B_{S_L}(\0)$ is shown in blue whereas the square $[-2L/N,2L/N]^2$ is in red.}
\label{fig1}
\end{center}
\end{figure}

On the event $E_{L,N}$, the cell generated by $\hx \in (\eta \cap {\hat R}_{L,N}) \cup \{\0_M \}$
is contained in $\cup \{Q_{L,j}, j \in I(\hx)\}$ and thus
\[
\sup_{\hx \in (\eta \cap {\hat R}_{L,N}) \cup \{\0_M\}} \Vol(C^{\rho_1}(\hx,(\eta \cap \hat{B}_{S_L}(\0)) \cup {\cal A})) \leq  \left(\frac{3L} {N}\right)^d.
\]
If $L \in (0, N t^{1/d}/3)$, then all cell volumes in ${\hat R}_{L,N}$ are at most $t$; thus
$\Delta^{\rho_1}(\infty) = 1$ on the event $E_{L_1,N}$ with $L_1 := \frac16  N t^{1/d}$.
Similarly,
$$
\inf_{\hx \in (\eta \cap {\hat R}_{L,N}) \cup \{\0_M\}}  \Vol(C^{\rho_1}(\hx, (\eta \cap \hat{B}_{S_L}(\0)) \cup {\cal A} \cup \{\0_M\})) \geq
\left(\frac{L}{3N}\right)^d.
$$
If $L \in (3N t^{1/d}, \infty)$, then all the cell volumes in ${\hat R}_{L,N}$ exceed $t$ and thus
$\Delta^{\rho_1}(\infty) = 0$ on the event $E_{L_2,N}$ with $L_2 := 6 N t^{1/d}$.
Taking $S := S_{L_1}\1\{E_{L_1,N}\} + S_{L_2}\1\{E_{L_2,N}\}$,
we have found two disjoint 
events $E_{L_1,N}$ and $E_{L_2,N}$, each  having positive probability, such that $\Delta^{\rho_1} (\infty)$ takes different values on these events,
and  thus it is non-degenerate.
Hence, $\sigma^2(\varphi^{\rho_1})>0$ and we can apply Theorem \ref{Thm1}(ii).

To prove the positivity of $\sigma^2(\varphi^{\rho_2})$  and  $\sigma^2(\varphi^{\rho_3})$ we shall consider a subset of $E_{L,N}$.  Assume
there exists a parameter $\mu^* \in [0, \mu]$ and a small interval  $I_{\alpha}(\mu^*) \subseteq [0, \mu]$ for some $\alpha \geq 0$ such that
$\mathbb{Q}_{\mathbb{M}}(I_{\alpha}(\mu^*))>0$.  Define  $\hat{E}_{L,N}$ to be the intersection of $E_{L,N}$ and the event $F_{L,N, \alpha}$ that the Poisson points in $[-L/2,L/2]^d$ have marks
in $I_{\alpha}(\mu^*)$.  If $\alpha$ is small enough, then the Laguerre and Johnson-Mehl cells nearly coincide with the Voronoi cells on the event $\hat{E}_{L,N}$.
Consideration of the events $\hat{E}_{L_1,N}$ and $\hat{E}_{L_2,N}$ shows that
$\Delta^{\rho_2} (\infty)$ and $\Delta^{\rho_3} (\infty)$ are non-degenerate, implying that
 $\sigma^2(\varphi^{\rho_2})>0$  and  $\sigma^2(\varphi^{\rho_3}) > 0$.
Thus Theorem \ref{thmApp1} holds for the Laguerre and Johnson--Mehl tessellations.
\qed

\vskip.3cm
\noindent{\em Remark.}
In the same way, one can establish that Theorem \ref{thmApp1} holds for any $h$ taking the form
$$
h(K) = \textbf{1} \{g(K) \leq t\} \quad \text{or} \quad h(K) = \textbf{1} \{g(K) > t\}
$$
for $t \in (0, \infty)$ fixed and $g: \textbf{F}^d \to \R$, a scale dependent function. By scale dependent function we understand that $g(\alpha K) = \alpha^q g(K)$ for some $q \neq 0$ and all $K \in \textbf{F}^d$ and $\alpha \in (0, \infty).$ Examples of the function $g$ include
 (a) $g(K) := \mathcal{H}^{d-1}(\partial K)$,  (b) $g(K) := \text{diam}(K)$,  (c) $g(K) := \text{radius of the circumscribed ball of }K$, and
(d) $g(K) := \text{radius of the circumscribed ball of }K$.

\vskip.3cm

\noindent{\bf Proof of Theorem \ref{thmApp2}.}
The unbiasedness is again a consequence of Theorem \ref{unbias}(i). To prove the asymptotic normality, we need to check the $p$-moment condition for $\xi^{\rho_i}(\hx, \eta) :=
\mathcal{H}^{d-1}(\partial C^{\rho_i} (\hx, \eta))\1\{C^{\rho_i}(\hx,\eta) \text{ is bounded}\}$ and the positivity of  $\sigma^2(\xi^{\rho_i}),  i=1, 2, 3$.

First we verify the moment condition with $p=5$. Given any $\hx,\hy \in \hR$, we assert that $\E \H^{d-1}(\partial C^{\rho_i}(\hx,\eta \cup \{\hy\}))^5 \leq c < \infty$
for some constant $c$ that does not depend on $\hx$ and $\hy$.
From Proposition \ref{prop:stab} there is a random variable $R_{\hx}$ such that
\[
C^{\rho_i}(\hx,\eta \cup \{\hy\}) = \bigcap_{\hz \in (\eta \cup \{\hy\} \setminus \{\hx\}) \cap \hat{B}_{R_{\hx}}(x)} \mathbb{H}_{\hz}(\hx).
\]
As in Proposition \ref{lemm:radius} we find $D_{\hx}$ such that $C^{\rho_i}(\hx,\eta \cup \{\hy\}) \subseteq B_{D_{\hx}}(\hx)$.
Then
\begin{align*}
\H^{d-1}(\partial C^{\rho_i}(\hx,\eta \cup \{\hy\})) &\leq \sum_{\hz \in (\eta \cup \{\hy\} \setminus \{\hx\}) \cap \hat{B}_{R_{\hx}}(x)}
\H^{d-1}(\partial \mathbb{H}_{\hz}(\hx) \cap B_{D_{\hx}}(\hx)) \\
&\leq c_{i,d} D_{\hx}^{d-1} \eta(\hat{B}_{R_{\hx}}(x))
\end{align*}
for some constant $c_{i,d}$ that depends only on $i$ and $d$.
Using the Cauchy--Schwarz inequality we get
\[
\E \H^{d-1}(\partial C^{\rho_i}(\hx,\eta \cup \{\hy\}))^5 \leq  c_{i,d}^5 (\E D_{\hx}^{10(d-1)})^{1/2} (\E \eta(\hat{B}_{R_{\hx}}(x))^{10})^{1/2}.
\]
By the property of the Poisson distribution we have
\[
\E \eta(\hat{B}_{R_{\hx}}(x))^{10} = \E (\E (\eta(\hat{B}_{R_{\hx}}(x))^{10} \mid R_x)) = \E P(\Vol(B_{R_{\hx}}(x))),
\]
where $P(\cdot)$ is a polynomial of degree 10. Both $D_{\hx}$ and $R_{\hx}$ have exponentially decaying tails and the decay is not depending on $x$.
Therefore, $(\E D_{\hx}^{10(d-1)})^{1/2} (\E \eta(\hat{B}_{R_{\hx}}(x))^{10})^{1/2}$ is bounded and the moment condition is satisfied with $p=5$.

The positivity  of the asymptotic variance can be shown similarly as in the proof of Theorem \ref{thmApp1}. We will show it only for the Voronoi case, as the Laguerre and Johnson--Mehl tessellations can be treated similarly. We will again find a random variable $S$ and a
$\Delta^{\rho_1}(\infty)$ such that for all finite ${\cal A} \subseteq \hat{B}_S(\0)^c$ we have
\begin{align*}
\Delta^{\rho_1}(\infty) &  = \sum_{ \hx \in (\eta \cap \hat{B}_S(\0)) \cup {\cal A} \cup \{\0_M \} }
  \xi^{\rho_1} (\hx, (\eta \cap \hat{B}_S(\0)) \cup {\cal A} \cup \{\0_M\} ) \\
 &   \hspace{1.5cm}   - \sum_{ \hx \in (\eta \cap \hat{B}_S(\0)) \cup {\cal A} } \xi^{\rho_1} (\hx, (\eta \cap \hat{B}_S(\0)) \cup {\cal A} )
\end{align*}
and moreover $\Delta^{\rho_1}(\infty)$ assumes different values on two events having positive probability and is thus non-degenerate. By Theorem 2.1 of \cite{PY1}, this is enough to show the positivity
of $\sigma^2(\xi^{\rho_1})$.

Let $L>0$ and let $N \in \N$ have odd parity. Abusing notation, we construct a collection of $N^d$ cubes $Q_{L,1}, \ldots, Q_{L,N^d}$ centered around $x_i \in \R^d, i= 1, \ldots, N^d$ such that
\begin{enumerate}
\item[(i)]  $Q_{L,i}$ has  side length $ \frac{L} {N}  $, and
\item[(ii)]  $\cup \{ Q_{L,i}, i=1, \ldots, N^d\} = [-\frac{L} {2}, \frac{L} {2}]^d$.
\end{enumerate}
There is a unique index $i_0 \in \{1, \ldots, N^d\}$ such that $x_{i_0} = \0$.
We define $\varepsilon_L, \hat{Q}_{L,i}$ and the event $E_{L,N}$ as in the proof of Theorem \ref{thmApp1}.
Note that under $E_{L,N}$
$$
\inf_{(x, m_x) \in \eta \cap \hat{Q}_{L,i_0} } \|x\| \leq \eps_L.
$$

Hence, on the event $E_{L,N}$, the insertion of the origin into the point configuration creates a new face of the tessellation whose  surface area is bounded below by $c_{min} (L/N)^{d-1}$
and bounded above by $c_{max} (L/N)^{d-1}$.
Thus
$$
c_{min} \left(\frac{L}{N}\right)^{d-1} + O\left(\varepsilon_L\left(\frac{L}{N}\right)^{d-2}\right)
\leq \Delta^{\rho_1}(\infty) \leq c_{max} \left(\frac{L}{N}\right)^{d-1} - O\left(\varepsilon_L\left(\frac{L}{N}\right)^{d-2}\right),
$$
where $O(\varepsilon_L\left(\frac{L}{N}\right)^{d-2})$ is the change in the combined surface areas of the already existing faces after inserting the origin. Events $E_{L_1,N}, E_{L_2, N}, L_1 <  L_2,$ both occur with positive
probability for any $L_1, L_2$.
Similarly as in the proof of Theorem \ref{thmApp1} we can find $N$, $S$, $L_1$ and $L_2$
($L_2-L_1$ large enough)
such that the value of $\Delta^{\rho_1}(\infty)$ differs on each event. Thus $\sigma^2(\xi^{\rho_1})$ is strictly positive.

To show that $\sigma^2(\xi^{\rho_2})$ and $\sigma^2(\xi^{\rho_3})$ are strictly positive we argue as follows. The Laguerre and Johnson-Mehl tessellations are close to the
Voronoi tessellation on the event $F_{L,N, \alpha}$, for $\alpha$ small.  Arguing as we did in the proof of Theorem \ref{thmApp1} and considering the event $\hat{E}_{L,N}$ given in the proof of that
theorem, we may conclude that $\sigma^2(\xi^{\rho_2})>0$  and  $\sigma^2(\xi^{\rho_3}) > 0$.  \qed

\section*{Acknowledgement}
The research of Flimmel and Pawlas is supported by the Czech Science Foundation, project 17-00393J, and by Charles University, project SVV 2017 No. 260454. The research of Yukich is supported by a Simons Collaboration Grant.  He thanks Charles University for its kind hospitality and support.

\end{document}